\documentclass{amsart}

\usepackage{amsfonts,amssymb,verbatim,amsmath,amsthm,latexsym,textcomp,amscd}
\usepackage{latexsym,amsfonts,amssymb,epsfig,verbatim}
\usepackage{amsmath,amsthm,amssymb,latexsym,graphics,textcomp}
\usepackage{graphicx}

\usepackage{url}
\usepackage{psfrag}
\usepackage{amsmath,amsthm,amsfonts,amssymb,mathrsfs,bm,graphicx,stmaryrd}
\usepackage{mathtools}
\usepackage[usenames]{color}
\usepackage{hyperref}
\usepackage [latin1]{inputenc}

\usepackage{tikz}
\usetikzlibrary{shapes,arrows}

\usepackage{cite}
\usepackage{xcolor}
\usepackage{tikz-cd}
\usepackage{float}
\usepackage{bm}
\usepackage{mathtools}
\usepackage{graphicx}
\usepackage{enumerate}
\usepackage{amsfonts,amssymb,verbatim,amsmath,amsthm,latexsym,textcomp,amscd,hyperref}
\usepackage{latexsym,amsfonts,amssymb,epsfig,verbatim}
\usepackage{graphicx}
\usepackage{color}
\usepackage{url}
\usepackage{pgfplots}
\usepackage{tikz}
\usepackage{pdfpages}
\graphicspath{ {./images/} }
\newtheorem{theorem}{Theorem}[section]
\newtheorem*{theorem*}{Theorem}
\newtheorem*{lemma*}{Lemma}
\newtheorem*{proposition*}{Proposition}
\newtheorem*{corollary*}{Corollary}
\newtheorem{lemma}[theorem]{Lemma}
\newtheorem{defn}[theorem]{Definition}
\newtheorem{prop}[theorem]{Proposition}
\newtheorem{cor}[theorem]{Corollary}
\newtheorem{example}[theorem]{Example}
\theoremstyle{definition}
\newtheorem{remark}[theorem]{Remark}

\newtheorem{convention}[theorem]{Convention}
\newtheorem{question}[theorem]{Question}
\newtheorem{criterion}[theorem]{Criterion}
\newtheorem{claim}[theorem]{Claim}

\def\ep{\epsilon}

\def\ri{\rightarrow}

\def\sse{\subseteq}
\def\gm{\gamma}
\def\bt{\beta}
\def\al{\alpha}

\def\pa{\partial}
\def\map{\rightarrow}

\newcommand\LMY{\lim^Y_{n \map \infty}}
\newcommand\LMX{\lim^X_{n \map \infty}}

\def\smallskip{\vspace\smallskipamount}
\def\medskip{\vspace\medskipamount}
\def\bigskip{\vspace\bigskipamount}

\def\V{\mathcal{V}}

\def\G{\mathcal{G}}

\def\D{\mathbb {D}}
\def\R{\mathbb {R}}
\def\C{\mathbb {C}}
\def\N{\mathbb {N}}

\def\Z{\mathbb {Z}}

\begin{document}
\title[Landing rays and ray Cannon--Thurston maps]{Landing rays and ray Cannon--Thurston maps}
\author{Rakesh Halder}  

\address{School of Mathematics, Tata Institute of Fundamental Research (TIFR) Mumbai, India}
\email{rhalder.math@gmail.com, rhalder@math.tifr.res.in}

\author{Mahan Mj}
\address{School of Mathematics, Tata Institute of Fundamental Research (TIFR) Mumbai, India}
\email{mahan@math.tifr.res.in}

\author{Pranab Sardar}
\address{Indian Institute of Science Education and Research (IISER) Mohali,Knowledge City,  Sector 81, S.A.S. Nagar 140306, Punjab, India}
\email{psardar@iisermohali.ac.in}

	

\begin{abstract}
In this paper, we describe a procedure to construct pairs of hyperbolic groups $H<G$ with the following properties.
\begin{enumerate}
\item  Every geodesic ray $\gamma$ in $H$  converges to a point 
$\xi_{\gamma}\in  \partial G$.
\item  The inclusion of $H$ into $G$ does not extend continuously to 
$\partial H \map \partial G$. In other words, a Cannon--Thurston map
does not exist for this pair of hyperbolic groups.
\end{enumerate}
Jeon, Kapovich, Leininger and Ohshika gave a property of conical limit points in the presence of a Cannon--Thurston map.
We convert this into a criterion for the existence of 
Cannon--Thurston maps and use it to prove 
the non-existence result in (2).  We obtain, in particular, a   geometric proof of  Baker--Riley's counterexample from \cite{baker-riley}.
\end{abstract}	

	\maketitle	

\setcounter{secnumdepth}{3}  
\setcounter{tocdepth}{3}     
	
		
\section{Introduction}\label{introduction} 
The aim of this paper is two-fold. First, we find the hyperbolic groups analog of the notion of landing rays in complex dynamics.
 Second, we find sufficient criteria to generate a family of examples of hyperbolic subgroups $G_1$ of hyperbolic groups $G$ for which
\begin{enumerate}
	\item  All geodesic rays in $G_1$ starting at the identity \emph{land},
	i.e.\ for all $\xi \in \partial G_1$, a  semi-infinite geodesic ray $[1, \xi)$ in a Cayley graph of $G_1$  accumulates on a unique point 
	of $\partial G$.
	\item  Nevertheless, the pair $(G_1,G)$ does not admit a Cannon--Thurston map.
\end{enumerate}
The following question was posed by the second  author in the first version of Bestvina's problem list.

\begin{question}\cite[Question 1.19]{bestvinaprob} (see also  \cite[p. 136]{mitra-trees},\cite[p. 354]{mitra-survey})\\\label{qn-main}
	Let $G_1<G$ be a pair of $(${\em Gromov}$)$ hyperbolic groups. Does the inclusion $i: G_1\ri G$ extend to a continuous map $\pa i: \pa G_1\ri \pa G$?
\end{question}

Such a continuous extension, if it exists, is known as a \emph{Cannon--Thurston}  map \cite{mitra-ct,mitra-trees}
after the pioneering work of Cannon and Thurston in \cite{CT,CTpub}.
If, for a pair  $G_1<G$ of hyperbolic groups, every geodesic ray in $G_1$ lands in the above sense, then we have a well-defined map
$\partial i_r: \partial G_1 \to \pa G$. Such a boundary extension $\partial i_r$ will be referred to as a \emph{ray Cannon--Thurston} map. Note that $\partial i_r$  \emph{need not be continuous}. In fact,
an aim of this paper is to provide examples where $\partial i_r$  exists but is not continuous. However, there are no  examples known to us of pairs $G_1<G$ where a 
ray Cannon--Thurston map does not exist (see Question~\ref{qn-rct} below).

Positive answers to Question~\ref{qn-main} were given in \cite{mitra-ct,mitra-trees}
in some special cases. Since then a number of results on the existence and structure of Cannon--Thurston maps have been proven. See for instance \cite{mitra-endlam,kl15,dkt,JKLO,baker-riley-hydra,ps-kap}
for a sample of results in the context of hyperbolic subgroups of hyperbolic groups. 
 We refer the reader  to  \cite{mahan-icm} for a  survey.  Baker and Riley \cite{baker-riley}  answered Question \ref{qn-main} negatively, by producing, for the first time, a  free subgroup $F$ of a hyperbolic group $G$ for which the inclusion $F \map G$ does \emph{not} admit  a Cannon--Thurston map. This  example was obtained using small cancellation theory starting  from a class of small-cancellation hyperbolic groups discovered by Rips \cite{rips}. Later, Matsuda and Oguni \cite{mats-oguni} used the work of Baker--Riley to show that any non-elementary hyperbolic group can be embedded in another hyperbolic group for which there is no Cannon--Thurston map. These were the only known examples  that answer Question~\ref{qn-main} negatively. In a different non-hyperbolic setup negative answers were obtained in \cite{radhika-etal,charney2024nonexistencecannonthurstonmapsmorse}.

 In this paper, we shall consider
 three primary classes of examples of hyperbolic subgroups $G_1$ of hyperbolic groups $G$. In all of these, the ambient hyperbolic group $G$ is of the form $G= H*_QL$, where
 \begin{enumerate}
 \item $Q$ is a  free group $\mathbb{F}_n$, $n\geq 2$.
 \item $L=Q*_\phi$ is a hyperbolic ascending HNN extension corresponding to a hyperbolic endomorphism $\phi$ of $Q$. (Note that $\mathbb{F}_2$ admits hyperbolic endomorphisms, but no hyperbolic automorphisms. Thus, for $\phi$ a hyperbolic endomorphism $n\geq 2$ suffices. For $\phi$ a hyperbolic automorphism, $n\geq 3$ becomes necessary.)
 \end{enumerate}
 
 Let $t$ denote the stable letter of  $L=Q*_\phi$. In all the three classes mentioned below, $G_1$ will be of the form $G_1 = K*\langle t \rangle$, where $K$ is an appropriately chosen subgroup of $H$.
 
 The pair $(G_1,G)$ in the three principal different classes of examples discussed in this paper thus depends on  three different classes of pairs  $(K,H)$ of a hyperbolic group $H$ and a hyperbolic subgroup $K$
 as three different starting points.
 \begin{enumerate}
 \item $K$ is normal in $H$ with non-elementary quotient $H/K$: in this case, $Q$ is the image of a section of a 
 malnormal quasiconvex free subgroup of $H/K$. This case is dealt with in Subsection~\ref{sec-normal}, particularly Theorem~\ref{extn case}
  (and Remark \ref{normal remk}).
 \item $K$ is commensurated by $H$: in this case  $Q$ is a malnormal quasiconvex free subgroup of $H$ intersecting $K$ trivially.  This case is dealt with in Subsection~\ref{sec-comm}, particularly Theorem~\ref{main thm for comm}.
 \item $K$ is free, and $H$ is a hyperbolic multiple ascending HNN extension
 of $K$, i.e.\ if $K=\langle x_1, \cdots, x_n\rangle$, $n \geq 2$, and $\phi_1, \cdots, \phi_m, m \geq 2$ are hyperbolic endomorphisms of $K$, then   a presentation of $H$ is given by
 $$\langle\{x_1, \cdots, x_n,t_1, t_2, \ldots, t_m\} |  \{t_jx_it^{-1}_j = \phi_j(x)^{-1}: \, 1\leq i\leq n, \,  1\leq j\leq m\}\rangle.$$
In this case,
  $Q < H$ is the free malnormal quasiconvex subgroup generated by the stable letters $\{t_1, t_2, \ldots, t_m\}$. This case is dealt with in Subsection~\ref{sec-endo}, particularly Theorem~\ref{thm-endo}.
 \end{enumerate}
 
 The following theorem combines the above cases into one omnibus statement.
 \begin{theorem}\label{thm-omni}
 Let  $G_1<G$ be a hyperbolic subgroup of a hyperbolic group constructed in one of the three above ways. Then
  \begin{enumerate}
 	\item   $(G_1,G)$ admits a ray Cannon--Thurston map. 
 	\item   $(G_1,G)$ does not admit a Cannon--Thurston map.
 \end{enumerate}
 \end{theorem}
 In fact, in the classes of examples given by 
 \begin{enumerate}
 \item normal subgroups (class 1), and
 \item commensurated subgroups  (class 2),
 \end{enumerate}  
  our techniques allow us to identify a set where the ray Cannon--Thurston map of Theorem~\ref{thm-omni} is continuous; see Theorem~\ref{extn case}
  (and Remark \ref{normal remk}) and Theorem \ref{main thm for comm}. 
  
 As mentioned earlier, Baker--Riley \cite{baker-riley} gave the first example of a pair $(G_1,G)$ not admitting a Cannon--Thurston map. 
 We revisit their example and give a new proof.

 \subsection{Techniques and criteria}\label{sec-tech}
 We now come to the techniques that go into proving Theorem~\ref{thm-omni}. \\
 
 \noindent \emph{Criterion for the existence of ray Cannon--Thurston maps:}\\
  The following theorem is the main result of Section \ref{sec-crit-land} which  provides a sufficient condition for the existence of
 ray Cannon--Thurston maps.

\begin{theorem}\label{prop-suffland-intro}

Suppose $H<G$ are hyperbolic groups such that the following hold:
\begin{enumerate}
\item $G=A\ast_C B$ and $H=U\ast_W V$
where $U<A, V<B, U\cap V=W$ and $H$ is an immersed subgraph
of subgroups for the amalgamated free product decomposition of $G$ as per 
\cite{bass-cov}. This amounts to saying that the natural map
$U/W\map A/C$ and $V/W \map B/C$ are both injective.
\item $A, B,C, U, V, W$ are all hyperbolic groups.
\item $C$ is malnormal and quasiconvex in $A$.
\item $W$ is malnormal and quasiconvex in $U$.
\item The inclusion maps $U\map G$ and $V\map G$ admit ray Cannon--Thurston maps.
\end{enumerate}
 
  Then the pair $(H,G)$ admits a ray Cannon--Thurston map.
 \end{theorem}

 \noindent \emph{Criterion for the non-existence of  Cannon--Thurston maps:}\\
 We describe a representative geometric criterion to be used to disprove the existence of Cannon--Thurston maps for a pair $(G_1, G)$
 (see Proposition~\ref{prop-jklo} for a more general statement). Let $i: \Gamma_1\subset \Gamma$ be an inclusion of Cayley graphs (with respect to finite generating sets) corresponding to the inclusion $i:G_1 \to G$
 (thus we are assuming implicitly that the generating set of $G$ includes that of $G_1$).
 
 \begin{criterion}\label{crit}  Let  $G_1 < G$ be a pair of hyperbolic groups as above.
  Let $[1,\xi)_{G_1}\subset \Gamma_1$ and $[1,\eta)_{G_1} \subset \Gamma_1$, with $\xi \neq \eta \in \partial G_1$ be landing rays accumulating at the same point $\zeta \in \partial G$. Further assume that $i([1,\xi)_{G_1}) \subset \Gamma$ is a quasigeodesic in $\Gamma$.
  Then $(G_1,G)$ does not admits a  Cannon--Thurston map.
 \end{criterion}

 A short self-contained deduction of a generalization of Criterion~\ref{crit} occupies Section~\ref{sec-crit}. We refer to Criterion~\ref{crit} and its variants as the JKLO criterion after Jeon, Kapovich, Leininger, and Ohshika, who deduced the key property that leads to this criterion in \cite{JKLO}.

Combining  Theorem~\ref{prop-suffland-intro} and 
the JKLO criterion leads to the main technical theorem of this paper (see Theorem~\ref{main thm}).
 
\begin{theorem}\label{main thm intro}
Let $H$ be a hyperbolic group. Let $F, K$ be two hyperbolic subgroups of $H$.
	Suppose, moreover, that the following hold.
\begin{enumerate}
\item $F$ is malnormal and quasiconvex in $H$.
\item $F\cap K=\{1\}$.

Further, suppose that $\phi:F\to F$ is a hyperbolic endomorphism of $F$. Let $G$ be the HNN extension $G=H\ast_{\phi}$ with stable letter $t$; and let
    $G_1$ be the subgroup of $G$ generated by $K\cup \{t\}$.

\item There is a sequence $\{y_n\}$ in $K$ converging to a point of $\pa K$ such that
		$\lim^G_{n\map\infty} y_n=\lim^G_{n\map\infty} t^n\in\pa G$.
	\end{enumerate}
    
	Then
	\begin{enumerate}
	\item $G_1, G$ are hyperbolic,
	\item $(G_1, G)$ is not a Cannon--Thurston pair.
	\item If a ray Cannon--Thurston map exists for the pair $(K,H)$ then a ray Cannon--Thurston map exists for the pair $(G_1,G)$.
    
\item The ray Cannon--Thurston map $\pa G_1\to\pa G$ $($ in $(3))$ is not continuous at the points in the $G_1$-orbits of $\lim^{G_1}_{n\to\infty}t^{-\infty}$ in $\pa G_1$.
    
\item If, moreover, $(K,H)$ is a Cannon--Thurston pair, then a ray Cannon--Thurston map 
for the pair $(G_1, G)$ is continuous on the complement of the $G_1$-orbits in $\partial G_1$ of the points $\lim^{G_1}_{n\map \infty}t^{\pm n}$.

\end{enumerate} 
\end{theorem}

Theorem~\ref{main thm intro} is deduced from 
 Theorem~\ref{prop-suffland-intro} and 
the JKLO criterion in Section~\ref{sec-main thm}.
Theorem~\ref{main thm intro} is then used to deduce the various specific cases occurring in Theorem~\ref{thm-omni}
(along with others) in Section~\ref{sec-applns}. A curious  complementarity of observed phenomena in the complex dynamics world versus the hyperbolic groups world is noted in Section~\ref{sec-cxdyn}.\\

\noindent {\bf Acknowledgments:} The authors are grateful to Xavier Buff,
Misha Lyubich, Curt McMullen, and Dennis Sullivan for useful correspondence pertaining to the state of the art in holomorphic dynamics.
The authors also thank Sabyasachi Mukherjee for helpful discussions on the same theme. Much of Section~\ref{sec-cxdyn} is based on these inputs. We also thank Mladen Bestvina for helpful correspondence. We thank  Harish Seshadri  for valuable inputs on a previous draft. Finally, MM would like to acknowledge that a key idea that inspired this project appeared first in RH's thesis done under the guidance of PS.

\section{Preliminaries}\label{prelim}
We refer the reader to \cite{gromov-hypgps}, \cite[Chapter III.H]{bridson-haefliger} and \cite{GhH} for background on hyperbolic groups. In this section we recall various notions relevant to the paper and fix  notation. 

Let $X$ be a metric space. $X$ is said to be {\em proper} if
any closed and bounded subset of $X$ is compact. An isometric action
of a (discrete) group $G$ on a metric space $X$ is called a {\em 
(metrically) proper action} if for all bounded set $A\subset X$,
the set $\{g\in G: A\cap gA \neq \emptyset\}$ is finite.

If $X$ is a metric space, $A\subset X$ and $x\in X$ then we define
$d(x,A)=\inf \{d(x,y):y\in A\}$; and for all $K\geq 0$ we define 
$N_K(A)=\{y\in X: d(y,A)\leq K\} $. We let $Hd(A,B)$ denote the {\em Hausdorff distance} between any two subsets $A, B$ of $X$ where 
$$Hd(A,B)=\inf\{ k: A\subset N_k(B), \, B\subset N_k(A)\}.$$

A {\em geodesic} in a metric space $X$ is an isometric embedding
$\alpha:I \map X$ where $I$ is an interval in $\R$, typically
a closed interval. If $I=[a,b]$ and $\alpha(a)=x, \alpha(b)=y$
then we say that $\alpha$ is a geodesic joining $x,y$. 
A geodesic segment joining $x$ and $y$ will often be denoted by $[x,y]_X$ or by $[x,y]$ when $X$ is understood. 

A metric space $X$ is called a {\em geodesic metric space} if any
pair of points in $X$ can be joined by a geodesic. All the metric spaces
that we work with in this paper are geodesic metric spaces. On a
subspace $Y$ of a geodesic metric space $X$ we always put the
{\em induced length metric} (see \cite[Chapter I.3]{bridson-haefliger})
from $X$.

Let $k\ge1$ and $\ep\ge0$. A map $f:X\map Y$ between two metric spaces is said to be a $(k,\ep)$-{\em quasi-isometric embedding $($or a $(k,\epsilon)$-qi embedding$)$} if for all $x,x'\in X$, we have $$\frac{1}{k}d_X(x,x')-\ep\le d_Y(f(x),f(x'))\le kd_X(x,x')+\ep.$$
By a $k$-{\em qi embedding}, we mean a $(k,k)$-qi embedding. We say $f$  is a qi embedding if $f$ is a $(k,\epsilon)$-qi embedding for some $k\geq 1$ and $\epsilon \geq 0$. 
A map $\al:I\sse\R\map X$ is said to be a $(k,\epsilon)$-quasigeodesic if $I$ is an interval in $\R$ and $\al $ is a $(k,\epsilon)$-qi embedding. 
A $(k,k)$-quasigeodesic is called a $k$-quasigeodesic.
When working with a quasigeodesic $\al:I\sse\R\map X$, we often conflate the map $\alpha$ and its image, and think of the quasigeodesic as a subset of $X$. A quasigeodesic
$\alpha:I\map X$ is called a {\em quasigeodesic ray} if $I$ is an interval of the form $[a, \infty)$ or $(-\infty, b]$. When $I=\R$ then $\alpha$ is called {\em a quasigeodesic line}. We note that $(1,0)$-quasigeodesics are simply
geodesics.

For a graph $\mathcal G$, we shall denote the vertex set and the edge set of
$\mathcal G$ by $\V(\mathcal G)$ and $\mathcal E(\mathcal G)$ respectively.

\subsection{Hyperbolic groups and related notions}\label{sec-ct}

In this paper, all the hyperbolic metric spaces that we work with are either trees or proper geodesic metric spaces. Also we adopt Rips' definition of hyperbolicity using slimness of geodesic triangles. 

\subsubsection{Gromov boundary}
The (Gromov) boundary of a hyperbolic space $X$, consisting of equivalence classes of asymptotic geodesic rays in $X$, will be denoted by $\partial X$.

Let $X\cup \partial X=\overline{X}$ denote the Gromov compactification. A geodesic ray $\gamma$ starting at $x \in X$ and accumulating at $\xi\in \partial X$
is denoted by $[x, \xi)$ and is said to join $x$ to $\xi$. 
In this case, we write $\gamma(\infty)=\xi$. Similarly,  a
geodesic line $\alpha$ accumulating at $\xi^{-}, \xi^{+}\in \partial X$ is denoted as
$(\xi^-,\xi^+ )$ and is said to join $\xi^{-}$ and $\xi^{+}$.

Existence of geodesic rays or  geodesic lines are guaranteed by the following.

\begin{lemma}\label{basic bdry lemma}\cite[p. 427-428]{bridson-haefliger}
	Suppose $X$ is either a proper hyperbolic space or a tree.
	Then for any $\xi \in \partial X$ and $p\in X\cup \partial X$, there is
	a geodesic ray or a geodesic line, according as $p\in X$ or $p\in \pa X$, joining $\xi$ and $p$.
	
	Moreover, if $X$ is $\delta$-hyperbolic and if
	 $\alpha$, $\beta$ are geodesic rays or geodesic lines joining the same pair of points of 	$X\cup \partial X$,  then $Hd(\alpha, \beta)\leq 2\delta$.
\end{lemma}

We recall the following basic property of the topology 
on $\bar X$ that is relevant.

\begin{lemma}\label{limit1}
	Suppose $\{x_n\}$ is a sequence in $X$ and $\xi\in \partial X$. Then
	$x_n\map \xi$ if and only if $d(x, [x_n,\xi))\map \infty$.
\end{lemma}
An important remark on notation is in order.
\begin{remark}
If $x_n\map \xi$ as in the lemma above then we write
$\LMX x_n=\xi$. 

When the point $\xi$ is not important then we simply write
$\LMX x_n  \in \pa X$ to indicate that $\{x_n\}$ is an unbounded
sequence of points in  $X$ which converges to a point of $\pa X$.

This convention is useful specially when one deals with a
pair of spaces $Y\subset X$ and a sequence $\{y_n\}$ in $Y$, and
one has to consider the limits of $\{y_n\}$ in both $\overline{X}$ and $\overline{Y}$. In such cases, we shall use the notation $\LMX y_n$ and $\LMY y_n$ to indicate where the limit is taken.

For a hyperbolic group $G$ and a sequence $\{g_n\}$ in $G$, the
notation $\lim^G_{n\map \infty} g_n\in \pa G$ or 
$\lim^G_{n\map \infty} g_n=\xi\in \pa G$ will mean the
limit is taken in a Cayley graph of $G$.
\end{remark}

We recall that the notation $Comm_G(H)$ for groups $H<G$ stands for the commensurator of $H$ in $G$:
$$Comm_G(H)=\{g\in G: H\cap gHg^{-1} \,\mbox{has finite index in both}\, H\, \mbox{and}\, gHg^{-1}\}.$$
We conclude this subsection with the following lemma, whose proof is straightforward.
\begin{lemma}\label{limit g^n}
	Suppose $G$ is a hyperbolic group and $\gamma$ is a quasigeodesic ray
	in (a Cayley graph of) $G$ starting at 1. Let $g_n=\gamma(n)$ for all $n\in \mathbb N$ and
	$\xi=\gamma(\infty)\in \pa G$. Suppose $h \in G$ is such that
	$d_G(1,g_nhg^{-1}_n)\map \infty$ as $n\map \infty$. Then
	$\lim^G_{n \map \infty}  g_nhg^{-1}_n  = \xi$.

	In particular, if $g, h\in G$ are infinite order elements and
	$h\not\in Comm_G(\langle g\rangle)$ then $\lim^G_{n\map \infty} g^nhg^{-n}=\lim^G_{n\map \infty} g^n$.
\end{lemma}

\subsubsection{Limit points and limit sets}\label{defn-ltset}
Suppose we have a hyperbolic metric space $X$ and a subset $Y\sse X$.
Then $\xi\in \partial X$ is called a {\em limit point} of $Y$ (in $\partial X$) if there is a sequence
$\{y_n\}$ in $Y$ such that  $\LMX y_n=\xi$. When a group $H$ acts by isometries on a hyperbolic metric space $X$, a limit point of $H$ is a limit
point of some (any) $H$-orbit in $X$.

The term \emph{limit set} refers to the set of all limit points.
The notation $\Lambda_X(Y)$ is used to denote the limit set of a
subset $Y$ of a hyperbolic space $X$.  For an action of a group $H$ by isometries on a hyperbolic $X$, the \emph{limit set}  of $H$ in $\pa X$, 
denoted by $\Lambda_X (H)$, is $\Lambda_X(H.x)$ for some (any) $x\in X$.
	
	A limit point $\xi\in\Lambda_X(Y) \sse \pa X$ of $Y$ is said to be a {\em conical limit point} of $Y$ if for some $($any$)$ (quasi)geodesic ray $\al$ in $X$ with $\al(\infty)=\xi$ there is $R\ge0$ such that $diam\big( Y\cap N_R(\al)\big)$ is infinite. For an action of a group $H$ by isometries on $X$,
	a conical limit point of $H$ is a conical limit point of an $H$-orbit.

For a hyperbolic group $G$ and a subgroup $H$ of $G$, the limit set of $H$ in $\pa G$ is defined to be the limit
set of $H$ for its action on a Cayley graph of $G$ and it is denoted by $\Lambda_G(H)$ or simply $\Lambda(H)$
when $G$ is understood.

\subsection{Cannon--Thurston maps, and ray Cannon--Thurston maps}\label{sec-lam} 
In this paper we introduce the notion of a \emph{ray Cannon--Thurston map} which we describe below.
However, let us first recall the well-known notion of a Cannon--Thurston map. Although this can be defined
for more general maps between hyperbolic spaces, the following is the most common.

\begin{defn}\label{CT-map}
\begin{enumerate}
\item
\cite{mitra-trees}
Suppose $Y\subset X$ are (proper) hyperbolic metric spaces where $Y$ is given the induced length metric from $X$. 
Let $i:Y \map X$ denote the inclusion map. A Cannon--Thurston map for this pair
is a map $\partial i : \partial Y \to \partial X$ with the following property:

For each $\xi\in\pa Y$ and for any sequence $\{y_n\}\sse Y$ with $\LMY y_n=\xi$ we have  $\LMX i(y_n)= \pa i(\xi)$. 

When a Cannon--Thurston map exists for a pair of spaces $Y\subset X$, we say that the pair $(Y,X)$ admits a \emph{Cannon--Thurston map}.

\item	 Let $H<G$ be hyperbolic groups.
 We say that the pair $(H,G)$ admits a \emph{Cannon--Thurston map} 
 $\pa i:\pa H \map \pa G$ if for some (any) inclusion
 of Cayley graphs $\Gamma_H \map \Gamma_G$, defined with respect to finite generating sets of $H$  and $G$, a Cannon--Thurston map exists.
\end{enumerate}	
\end{defn}
If a Cannon--Thurston map exists for a pair of spaces $Y\subset X$ then we say that $(Y,X)$  is a \emph{Cannon--Thurston pair}. If a Cannon--Thurston
map exists for an inclusion of hyperbolic groups $H<G$ then we say that
$(H,G)$ is a Cannon--Thurston pair.	
\begin{defn}
Suppose $Y\subset X$ are hyperbolic metric spaces and $i:Y\map X$
is the inclusion map.
\begin{enumerate}
\item \label{defn-landingray} 
A (quasi) geodesic ray
$\gamma$ in $Y$ is called a {\em landing ray} for the pair $(Y,X)$ if the
pair of metric spaces $(\gamma, X)$ is a Cannon--Thurston pair, or equivalently if
$\LMX \gamma(n)\in \pa X$.

\item If all (quasi)geodesics rays of $Y$ are landing rays for the pair
$(Y,X)$ then we say that a {\em ray Cannon--Thurston} map exists for the pair
$(Y,X)$.

\item If a ray Cannon--Thurston map exists for any (some) inclusion of
Cayley graphs of $H<G$ with respect to finite generating sets, then we
say that $(H,G)$ admits a \emph{ray Cannon--Thurston map}.
\end{enumerate}

\end{defn}
If a ray Cannon--Thurston map exists for a pair of hyperbolic 
spaces $Y\subset X$ or a pair of hyperbolic groups $H<G$ then we say that
this is a {\em ray Cannon--Thurston pair}.

\begin{remark}
(1) Existence of a Cannon--Thurston map implies the existence of a ray Cannon-
Thurston map; but the converse is false in general, as the  main theorem of this paper shows.

(2) This is about notation. We denote inclusion maps
$Y\map X$ by $i:Y\map X$ or $i_{Y,X}:Y\map X$. If $(Y,X)$ is a Cannon--Thurston pair then the corresponding Cannon--Thurston map will be denoted by
$\pa i:\pa X\map \pa Y$, or by $\pa i_{Y,X}:\pa Y \map \pa X$.

Whenever $(Y,X)$ is a ray Cannon--Thurston pair, then the corresponding ray Cannon--Thurston map
will be denoted by $\pa i_{r,Y,X}: \pa Y \map \pa X$ or simply by $\pa i_r:\pa Y\to\pa X$ when $Y$ and $X$ are understood.

Similarly for hyperbolic groups $H<G$, 
$\pa i_{H,G}:\pa H\map \pa G$ or $\pa i_r :\pa H \map \pa G$ will
bear similar meanings.

(3) If a Cannon--Thurston map exists for a pair of hyperbolic spaces
$Y\subset X$ (or a pair of hyperbolic groups $H<G$) then it follows 
from Lemma \ref{ray ct cont criteria} 
below that the Cannon--Thurston map is continuous. However, this is
false for ray Cannon--Thurston maps as the results of Section \ref{sec-main thm} of this article
will show.
\end{remark}

\begin{convention}
    If $Y\subset X$ are hyperbolic spaces such that a ray  Cannon--Thurston map
$\pa i_r:\pa Y\map \pa X$ exists but a Cannon--Thurston map does not exist for the
pair $(Y,X)$ then we will say that the pair $(Y,X)$ admits {\bf only a ray Cannon--Thurston map}.
A similar convention will be used for a pair of hyperbolic groups $H<G$.
\end{convention}

The following lemma gives a necessary and sufficient condition
for the existence of Cannon--Thurston maps. 
\begin{lemma}\label{mitra's criterion} \cite{mitra-trees}
Suppose $f:Y\map X$ is a map between two proper hyperbolic metric spaces.
Let $y_0\in Y$  be some base-point. Then $f$ admits a Cannon--Thurston map if and only if
 there is a function $\eta:\N \map\N$ (where $\eta(r)\map\infty$
as $r\map\infty$) such that the following holds:

For any $y,y'\in Y$, and any geodesic $\beta$ in $Y$ joining them and any geodesic $\alpha$ in $X$
joining $f(y),f(y')$, one has $$d_X(f(y_0),\al)\le r\, \mbox{ implies }\, d_Y(y_0,\bt)\le \eta(r).\hspace{1.5 cm} (*)$$
\end{lemma}

The following lemma is an analog of Lemma \ref{mitra's criterion}  for ray Cannon--Thurston maps.
We state this for completeness. The proof of this is similar to that of Lemma
\ref{mitra's criterion} in \cite{mitra-trees} and so we skip it.

\begin{lemma}\label{rmk-necsuff}{\em ( {\bf Criterion for ray Cannon--Thurston maps})}
Suppose $i:Y\map X$ is an inclusion of proper hyperbolic metric spaces and $y_0\in Y$. 
Then this pair of spaces admits a ray Cannon--Thurston map $\pa i_r :\pa Y \map \pa X$ 
if for all $\xi \in \pa Y$ there is a function $\eta_\xi: \N \to \N$ such that for any
geodesic ray $\alpha$ in $Y$ joining $y_0$ to $\xi$, and for any $p, q \in \alpha$,
$d_X(1, [p,q]_X) \leq n$ implies that either
$d_Y(y_0,p) \leq \eta_{\xi}(n)$ or $d_Y(y_0,q) \leq \eta_{\xi}(n)$.

 Moreover, a ray Cannon--Thurston map is a Cannon--Thurston map if and only if  
there is a function $\eta: \N \to \N$ such that $\eta_\xi(n) \leq \eta(n)$ for all $\xi\in \pa Y$
and all $n\in \N$.
\end{lemma}

We end this subsection with the following lemma. 
Since the proof follows by standard arguments we skip it.

\begin{lemma}\label{ray ct cont criteria}
    Suppose $(Y,X)$ is a ray Cannon--Thurston pair. Suppose $\xi \in \pa Y$ is such that for any
    sequence $\{y_n\}$ in $Y$ with $\LMY y_n =\xi$, we have $\LMX y_n =\pa i_r(\xi)$. Then $\pa i_r$ is continuous at $\xi$.

    In particular, existence of a Cannon--Thurston map implies that it is continuous.
\end{lemma}

\subsection{Quasiconvexity}
Suppose $K\ge0$. We recall that a  subset $A$ of a geodesic metric space $X$  is said to be $K$-{\em quasiconvex} if $[x,y]\sse N_K(A)$ for all $x,y\in A$ and for all geodesics $[x,y]$ joining $x,y$. We say that $A$ is {\em quasiconvex} in $X$ if $A$ is $K$-quasiconvex for some $K\ge0$. In a hyperbolic space, due to stability of quasigeodesics \cite[Theorem 1.7]{bridson-haefliger}, this becomes a qi invariant notion. In particular, quasiconvexity of a subgroup $H$ of a hyperbolic group $G$ is independent of the choice of Cayley graph of $G$. 

\subsubsection{Two lemmas on limits}
\begin{defn} 
If $\{A_n\}$ is a sequence of subsets in a hyperbolic space
$X$, then we say that {\em $\{A_n\}$ converges to $\xi \in\pa X$} if the following holds:\\
Let $x\in X$ be any point. Then
for all $M\in \N$ there is $N=N(M) \in \N$ such that for any $n>N$ and 
any $x_n\in A_n$ we have $d(x,[x_n,\xi))>M$.
\end{defn}
It would follow that $\{A_n\}$ converges to $\xi \in \pa X$ if and only if  any neighborhood of $\xi$ in $\overline{X}$ contains all but finitely 
many $A_n$'s. However,
by Lemma \ref{limit1} this means, in particular, that for any sequence
$\{x_n\}$ in $X$ where $x_n\in A_n$, for all $n\in \N$, we have 
$\LMX x_n =\xi$. 

The following lemmas are very standard, hence we state
them without proof.
\begin{lemma}\label{set conv}
Suppose $X$ is a hyperbolic metric space and $\{A_n\}$ is a sequence
of $k$-quasiconvex sets in $X$, where $k\ge0$. Let $x\in X$. Suppose $d(x, A_n)\map \infty$
as $n\map \infty$. Finally, suppose that there is a sequence $\{x_n\}$ in $X$ where $x_n\in A_n$ for all $n\in \N$ and $\LMX x_n =\xi\in \pa X$.
Then $\{A_n\}$ converges to $\xi$.
\end{lemma}

\begin{lemma}\label{limitset intersection}
    Suppose $X$ is a proper hyperbolic metric space and $\{A_n\}$ is a descending sequence of nonempty, $k$-quasiconvex sets, i.e. $A_{n+1}\subset A_n$ for all $n\in \N$, where $k\ge0$.     If $d(x, A_n)\map \infty$ for some $x\in X$ and 
    $\Lambda(A_n)\neq \emptyset$ for all $n\in \N$ then $\cap \Lambda(A_n)$ is a point of $\pa X$ (and $\{A_n\}$ converges to this point).
\end{lemma}

 \subsubsection{Malnormal quasiconvex subgroups of hyperbolic groups}\label{subsec-malnormalqc} 

We recall that a subgroup $F$  of a group $H$ is called {\em malnormal (resp.\ almost malnormal) }	if for all $h \in H \setminus F$, $hFh^{-1} \cap F$ is $\{1\}$	(resp.\ finite).

The following lemma is immediate from the definition.
\begin{lemma}\label{composition mal}
	Suppose $K<H<G$ are groups such that $K$ is malnormal $($resp. almost malnormal$)$ in $H$ and $H$ is malnormal $($resp. almost malnormal$)$ in $G$. Then $K$ is malnormal $($resp. almost malnormal$)$ in $G$.
\end{lemma}

We shall make use of the following basic lemma a number of times later on.
The proof is easy and hence we omit it.
\begin{lemma}\label{malnormal basic}
	Suppose $G$ is any group which has no subgroup isomorphic to $\mathbb Z\oplus \mathbb Z$.
	Suppose $N\lhd G$ is a torsion-free subgroup. Let $Q=G/N$ and $\pi:G\map Q$ be the quotient map.
	Suppose  $H$ is a malnormal torsion-free subgroup of $Q$. Suppose $g:H\map G$ is a group
	theoretic section of $\pi$ over $H$. Then $g(H)$ is a malnormal subgroup of $G$.
\end{lemma}

 The following theorem shows that  (almost) malnormal, quasiconvex subgroups are
 ubiquitous in hyperbolic groups.

 \begin{theorem}\label{main malnormal thm}(\cite{kapovich-nonqc}, \cite[Theorems 2.24 (c), 2.10]{DGO})
	Suppose $G$ is a hyperbolic group and $H<G$ is a quasiconvex subgroup. Moreover, we assume that $H\simeq F_n$ for $n\ge2$. Then given $m\in\N$, we can find a free subgroup $F<H$ of rank $m$ such that $F\times A<G$ where $A$ is a finite subgroup of $G$ and $F\times A$ is almost malnormal, quasiconvex in $G$.
\end{theorem}

When $G$ is torsion-free, the above result is  due to Ilya Kapovich (\cite{kapovich-nonqc}). The general case follows from the work of
Dahmani, Guirardel and Osin (\cite[Proposition 2.10, Theorems 2.24 (c)]{DGO}).


\subsubsection{A quasiconvex embedding lemma}
The following lemma is well known, e.g. it is implied in \cite[Proof of Theorem 1.1, pp. 44--45]{mats-oguni}. We include a proof for the sake of completeness.

\begin{lemma}\label{mat-oguni elaborate}
	Suppose $G=H_1*_KH_2$ is an amalgamated free product where $H_1$, $H_2$ are hyperbolic, and $K$ is an almost malnormal and quasiconvex subgroup of $H_1$.  Then $G$ is hyperbolic, and $H_2$ is almost malnormal and quasiconvex in $G$.
\end{lemma}

\begin{proof} 
Since $K$ is almost malnormal and quasiconvex in $H_1$,  $H_1$ is hyperbolic relative to $K$ by \cite[Theorem 7.11]{bowditch-relhyp-publish}. Then by \cite[Theorem 0.1 (2)]{dahmani-com}, $G$ is hyperbolic relative $H_2$. Now, since $H_2$ is hyperbolic, 
it follows by \cite[Corollary 2.41]{osin-2006} that $G$ is hyperbolic. 
	
	Finally, as $G$ is hyperbolic relative to $H_2$,  by \cite[Theorem $1.5$]{osin-bddgen} $H_2$ is quasiconvex and almost malnormal in $G$.
	\end{proof}

\begin{defn}\label{defn-hyp endo}
  For a hyperbolic group $Q$, an  endomorphism $f:Q\map Q$
is said to be a {\em hyperbolic endomorphism} if (1) $f$ is injective,
(2) $f(Q)$ is quasiconvex in $Q$ and (3) the corresponding
HNN extension $Q\ast_f$ is hyperbolic.  
\end{defn}

\begin{remark}
In Definition \ref{defn-hyp endo}  conditions $(1)$ and $(3)$ imply that $Q$ is a free product of a free group and surface groups if one assumes that $Q$ is torsion-free. This is essentially by Paulin's theorem and \cite{rips-sela} as explained in the introduction to \cite{mitra-endlam} (the argument there works for endomorphisms as well as automorphisms). Therefore, in this case, \cite[Corollary 2]{kap-qc-amal} implies that condition $(2)$ follows from remaining hypotheses.
\end{remark}

\begin{cor}\label{basic combination}
	Suppose $H$ is a hyperbolic group and $Q$ is an almost malnormal quasiconvex subgroup of
	$H$. Suppose $\phi:Q\map Q$ is a hyperbolic endomorphism of $Q$
	and $L=Q*_\phi$ is the ascending HNN extension corresponding to $\phi$ with  stable letter $t$ (so that $tqt^{-1}=\phi(q)$ in $L$ for all $q \in Q$). 	Let $G=H*_{Q}L$.
	Then \begin{enumerate}
		\item $G$ is hyperbolic. 	\item $L$ is quasiconvex in $G$.
		\item $\lim_{n\map \infty}^L t^n\in \Lambda_L(Q)$.
		\item $\lim_{n\map \infty}^G t^n\in \Lambda_G(Q)$.
	\end{enumerate}
\end{cor}

\begin{proof}
	Since $G=H*_{Q}L$ and $L$ is hyperbolic, conclusions $(1)$ and $(2)$ follows from Lemma \ref{mat-oguni elaborate}.

	$(3)$  For any $q \in Q$ and $n\in \mathbb N$, we have $t^n q t^{-n} = \phi^n(q) \in Q$. We note that $\{\phi^n(q)\}$ is an unbounded sequence in $Q$ since $\phi$ is hyperbolic. Hence, by Lemma \ref{limit g^n}, we have $\lim_{n\map \infty}^L t^n\in \Lambda_L(Q)$.

	(4) Since $L$ is quasiconvex in $G$, so by  $(3)$, we have $\lim_{n\map \infty}^G t^n\in \Lambda_G(Q)$.
\end{proof}

\subsection{Trees of spaces for amalgamated free products of groups}\label{trees of spaces}
Graphs of hyperbolic groups satisfying a qi embedded condition
were introduced by Bestvina and Feighn in \cite{BF} to which we refer for 
the relevant notion of `trees of metric spaces'. 
We  refer to \cite{BF} or \cite{scott-wall} for 
details. However, we recall below a coarsely equivalent
object- which may be called `a tree of metric graphs',  
following \cite{ps-conical}. Moreover, this is described only for
amalgamated free products of groups as this is the only
case relevant for us. For a graph $\mathcal G$,
recall that $\mathcal V(\mathcal G)$ and $\mathcal E(\mathcal G)$ denote the vertex set
and edge set of $\mathcal G$ respectively.

Suppose $H=U\ast_ W V$ is an amalgamated free product of finitely
generated groups $U, V$ along a common finitely generated  subgroup $W$. 

\noindent {\bf Bass--Serre tree $T_H$.}\\
The Bass--Serre tree $T_H$ is defined as follows.
$$\V(T_H)= H/U \bigsqcup H/V, \,\, \mathcal E(T_H) = H/W$$
where $hW\in E(T_H)$ joins $hU, hV\in \mathcal V(T_H)$.

 As
$H$ naturally acts on $H/U, H/V, H/W$ these actions induce
a natural action of $H$ on $T_H$. 
  We shall need
finite generation of the groups involved for the next construction.

\smallskip \noindent
{\bf Tree of metric graphs $\pi_H:X_H\map T_H$.}\\
Let $S_W$ be a finite generating set for $W$, and let $S_U \supset S_W$
and $S_V \supset S_W$ be finite generating sets for $U$ and $V$
respectively. 
Now let $S=S_U \cup S_V$. Then $S$ is a finite generating set
for $H$. Let $\Gamma(H,S)$ be the Cayley graph of $H$ with respect to
$S$. 

Let $X_1=\Gamma(U,S_U)$, $X_2=\Gamma(V,S_V)$ and $X_{12}=\Gamma(W,S_W)$
be the Cayley graphs of $U, V, W$ respectively with respect their
chosen generating sets. Note that these graphs are naturally subgraphs
of $\Gamma(H,S_H)$. Thus for all $h\in H$, $hX_1, hX_2, hX_{12}$
are subgraphs of $\Gamma(H, S_H)$. The graph $X_H$ is obtained by introducing some new edges to the 
disjoint union of  the graphs
$$X'_H=(\bigsqcup_{hU\in H/U} hX_1) \bigsqcup (\bigsqcup_{hV\in H/V} hX_2)$$
as follows: 
For every edge $hW\in \mathcal E(T_H)$, each of the pair of  vertices 
$hw\in hX_1$ and $hw\in X_2$ are joined by an edge where $w\in W$.

We note that (1) there is a natural map $\pi_H: X_H \map T_H$; 
(2) the natural $H$-action on $\Gamma(H,S)$
induces an $H$-action on $X_H$ so that $\pi_H$ is $H$-equivariant.

\begin{lemma}\textup{(\cite[Lemmas 3.4, 3.5]{ps-conical}}\label{trees of spaces lemma}
(1) The natural action of $H$ on $X_H$ is proper and cofinite.

(2) 
Let $x_0$ correspond to the vertex $1\in U\subset X_H$ so that $\pi_H(x_0)=U\in H/U$. Let $\phi:H\map X_H$ be the orbit map $h\mapsto hx_0$. Then there exists
$D>0$ such that $Hd(hX_1, hU.x_0)\leq D $, $Hd(hX_2, hV.x_0)\leq D $ and
$Hd(hX_{12}, hW.x_0)\leq D $ where 
$hX_{12}\subset hX_1$ for all $h\in H$.
\end{lemma}
We note that $H$ is quasi-isometric to $X_H$ by the \v{S}varc--Milnor
lemma (\cite[Lemma 8.9, I]{bridson-haefliger} and Lemma \ref{trees of spaces lemma} (1).

{\bf Convention.} (1) We shall refer to the
vertices coming from $H/U$ as the {\em $U$-type vertices} of $T_H$ and similarly
those from $H/V$ will be referred to as the {\em $V$-type vertices} of $T_H$.

(2) The inverse image under $\pi_H$ of a vertex of $T_H$ will be referred to as
a {\em vertex space} of the tree of spaces or simply a vertex space of $X_H$.
Hence these are of the form $hX_1=\pi^{-1}_H(hU)$ or $hX_2=\pi^{-1}_H(hV)$ 
for some $h\in H$. For the edge in $T_H$ connecting $hU$ and $hV$, the subgraph
$hX_{12}$ in $hX_1$ (or $hX_2$) will be referred to as the {\em edge space
in $hX_1$ (resp. in $hX_2$)}.

\section{An existence theorem for ray Cannon--Thurston maps}\label{sec-crit-land}

We prove the following theorem in this section some of whose ideas go back to \cite{GMRS}. 
 \begin{theorem}\label{prop-suffland}

Suppose $H<G$ are hyperbolic groups such that the following hold:
\begin{enumerate}
\item $G=A\ast_C B$ and $H=U\ast_W V$
where $U<A, V<B, U\cap V=W$ and $H$ is an immersed subgraph
of subgroups for the amalgamated free product decomposition of $G$ as per 
\cite{bass-cov}. This amounts to saying that the natural map
$U/W\map A/C$ and $V/W \map B/C$ are both injective.
\item $A, B,C, U, V, W$ are all hyperbolic groups.
\item $C$ is malnormal and quasiconvex in $A$.
\item $W$ is malnormal and quasiconvex in $U$.
\item The inclusion maps $U\map G$ and $V\map G$ admit ray Cannon--Thurston maps. 
\end{enumerate}
 
  Then the pair $(H,G)$ admits a ray Cannon--Thurston map.
 \end{theorem}
\begin{proof}
During the course of the proof, for any finitely generated
group $P$,  $S_P$ will denote a chosen finite generating set for $P$. For the sake of
clarity of exposition, the proof is divided into three steps.

\smallskip
\noindent
{\bf Step 1. Construction of trees of spaces.} 
The proof uses the geometry of the trees of graphs $\pi_G:X_G\map T_G$ and $\pi_H:X_H\map T_H$
associated with the amalgamated free product decompositions of $G,H$ respectively, as was described
in Section \ref{trees of spaces}. We note that $T_G$, $T_H$ are the Bass--Serre trees which have
nothing to do with the generating sets of the groups involved; however, the graphs $X_G$ and $X_H$ 
are dependent on choices of generating sets $S_G, S_H$ of $G$ and $H$ respectively. 
We assume that the following hold: 
\begin{enumerate}
    \item $S_U\subset S_A $, $S_V\subset S_B$
    \item $S_W\subset S_C$
    \item $S_W\subset S_U\cap S_V$, $S_C\subset S_A\cap S_B$.
    \end{enumerate} 
Then we define $S_G=S_A\cup S_B$, $S_H=S_U\cup S_V$, so that $S_H\subset S_G$. It follows from
the construction of Section \ref{trees of spaces} and hypothesis (1) of the theorem that 
there are natural  $H$-equivariant embeddings $\phi:X_H\map X_G$ and $\psi:T_H\map T_G$
such that $\pi_G\circ \phi=\psi\circ \pi_H$. 
Without loss of generality, therefore, we will assume that
(1) $X_H\subset X_G$,  (2) $T_H\subset T_G$ and (3) $\pi_H$ is the restriction of $\pi_G$ on $X_H$.

 By Lemma \ref{trees of spaces lemma} (1)
 it is now enough to show that any geodesic ray
in $X_H$ converges to a point of $\pa X_G$.
This is done in the next two steps of the proof.
However, before we proceed further
we note the following observations about these trees of spaces.\\

\noindent
{\bf Properties of $X_G$ and $X_H$:}\\
$(i)$ We note that there are two types of vertices in $T_H$ $-$ the $U$-type and the $V$-type. Similarly,
there are two types of vertices in $T_G$ $-$ the $A$-type and the $B$-type. 
Moreover, any $V$-type vertex of $T_H$ is a $B$-type vertex
of $T_G\supset T_H$ and any $U$-type vertex of $T_H$ is an $A$-type vertex of $T_G$.\smallskip\\
$(ii)$ If $t\in \V(T_H)$ is a $V$-type vertex, say $t=hV$,
the pre-image $\pi^{-1}_H(t)$, by definition, is the subgraph $h\Gamma(V, S_V)$ of $\Gamma(H,S_H)$. By
Lemma \ref{mat-oguni elaborate} $V$ is quasiconvex in $\Gamma(H,S_H)$ whence each coset of $V$ in 
$\Gamma(H,S_H)$ is (uniformly) quasiconvex. Then it follows that $\pi^{-1}_H(t)$ is (uniformly) quasiconvex
in $X_H$ by Lemma \ref{trees of spaces lemma}.

Similarly, $\pi^{-1}_G(t')$ is (uniformly) quasiconvex in $X_G$ for any
$B$-type vertex $t'$. In particular, this is true if $t'$ were a $V$-type vertex of $T_H$.\smallskip\\
$(iii)$ If $t\in \V(T_H)$ is a $U$-type vertex then by  the tree of spaces structure of $X_H$, it follows that $\pi^{-1}_H(B(t; 1))$
is a (uniformly) quasiconvex subset of $X_H$.

One gets a similar statement for $X_G$ as well for an $A$-type vertex.\smallskip

\noindent {\bf Step 2. A classification of quasigeodesics in $X_H$.} The aim of this 
part of the proof is to establish the following.

{\bf Claim:} {\em For any geodesic ray $\gamma \subset X_H$ there is a
uniform quasigeodesic ray $\gamma'\subset X_H$ asymptotic to $\gamma$ such that exactly 
one of the following holds:\\
(1) $\gamma'$ is contained in a $V$-type vertex space. \\
(2) The image of $\pi_H \circ\gamma'$ is contained in the $1$-neighborhood of an $U$-type vertex. \\
(3) $\pi_H \circ\gamma'$ is a metrically proper map
i.e. inverse images of bounded sets under $\pi_H \circ\gamma'$ are bounded; moreover, the image 
of $\pi_H \circ\gamma'$ is contained in the $1$-neighborhood of a geodesic ray of $T_H$.} \smallskip

 {\em Proof of claim:} There are two cases to consider. First,
suppose that there is $t\in \V(T_H)$ such that $\pi_H \circ \gamma (n)=t$ for infinitely many $n\in\N$. 
In this case, there are two possibilities. If $t$ is a $V$-type 
vertex then a subsequence of $\{\gamma(n)\}$ is contained in $\pi^{-1}_H(t)$ 
which is quasiconvex in $X_H$  by property $(ii)$ above. Thus we may find a uniform quasigeodesic $\gamma'$ of $X_H$ asymptotic to $\gamma$
which is contained in $\pi^{-1}_H(t)$. 
However,  if $t$ is a $U$-type vertex then, in the same way, by property $(iii)$ above, we may 
find a uniform quasigeodesic $\gamma' \subset \pi^{-1}_H(B(t; 1))$ which is asymptotic to $\gamma$.

In the complementary case, the image of $\pi_H \circ \gamma $
is a subtree of $T$ in which every vertex has finite degree and the image
is of infinite diameter. We note that any unbounded, proper hyperbolic metric space
contains a geodesic ray; e.g. see the proof of Lemma $3.1$ in \cite[Chapter III.H]{bridson-haefliger}.
Thus there is a geodesic ray $\alpha$, say, contained
in the  image of $\pi_H \circ \gamma $. We note that on any geodesic in $T_H$, $V$-type and
$U$-type vertices appear alternately.
Suppose $\{t_n\}$ is the sequence of consecutive vertices
on $\alpha$. Then without loss of generality, we may assume that each $t_{2n-1}$
is a vertex of $V$-type for all $n\in \N$. 
Let $x_n\in \gamma \cap \pi^{-1}_H(t_{2n-1})$. For all $n\in \N$,
let $\gamma'_n$ be a 
uniform quasigeodesic of $X_H$ joining $x_n, x_{n+1}$  such that 
$\pi_H(\gamma'_n)\subset B(t_{2n}; 1)$. This can be ensured  as
 $\pi^{-1}_H(B(t_{2n}; 1))$ is a uniformly quasiconvex subset of $X_H$ by $(iii)$ above.
 Let $\gamma'$ be the concatenation of the $\gamma'_n$'s in this case.
 Then it follows that $\gamma'$ is of type (3) as mentioned in the claim.\\
{\bf Step 3. Completion of the proof.}
To complete the proof of the theorem, it is enough to show that any $\gamma'$
as in Step 2 converges to a point of $\pa X_G$. The proof of this is divided into two cases. We shall continue with the notation of Step 2.

{\bf Case 1.} $ \gamma'$ is of type (3).\\
We recall that $t_{2n}$ is a $U$-type vertex of $T_H$ 
(and an $A$-type vertex of $T_G$) on $\alpha$ for all $n\in \N$.
Suppose $T_n$ is the component of $T_G$ containing $t_{2n+1}$
obtained by removing an interior of the edge joining $t_{2n}, t_{2n+1}$. 
Let $\mathscr I_n=\pi^{-1}_G(T_n)$. By property $(ii)$ of $X_G$ mentioned in the
Step $1$ of the proof, $\pi^{-1}_G(t_{2n+1})$ is uniformly quasiconvex
in $X_G$. It then follows from the tree of spaces structure of $X_G$
that each $\mathscr I_n$ is uniformly quasiconvex in $X_G$.
Moreover, for any point $x\in X_G$, $d(x, \mathscr I_n)\map \infty$ as $n\map \infty$. Now,
as  we have a descending chain of uniformly quasiconvex and unbounded sets
$\mathscr I_1 \supset \mathscr I_2\supset \cdots $ we have
$\cap _n\Lambda(\mathscr I_n)=\{\xi\}$ for some $\xi\in \pa X_G$,
by Lemma \ref{limitset intersection}. It follows that for any sequence
$\{x_n\}$ in $X_G$ such that $x_n\in \mathscr I_n$ for all $n\in \N$,
we have $x_n\map \xi$ as $n\map \infty$. From this, it immediately
follows that $\gamma'(n)\map \xi$ as $n\map \infty$.

{\bf Case 2.} $\gamma'$ is of type (1) or (2).\\
In this case, if $\gamma'$ has an infinite segment which 
is contained in a vertex space of $U$-type or $V$-type then we are done
as $(U,G), (V,G)$ are ray Cannon--Thurston pairs. Else there is a $U$-type vertex $t$, say, 
such that $(a)$ $\pi_H(\gamma') \subset B(t;1)$, and $(b)$ 
$\gamma'\setminus \pi^{-1}_H(t)$ has infinitely many components
of unbounded length each contained in a $V$-type vertex space in $\pi^{-1}_H(B(t;1))$. In this case,
the idea is to replace $\gamma'$ by another quasigeodesic in $X_H$ which is asymptotic
to $\gamma'$ and to show that the new quasigeodesic converges to
a point of $\pa X_G$.

As $H$ acts transitively on the $U$-type vertices of $T_H$,
without loss of generality, we may assume that $t$ is the vertex corresponding to the identity coset of $U$ in $H$. 
Let $X_1$ denote the corresponding vertex space of $X_H$. 
Let $X_{12}\subset X_1$ be the edge space corresponding to the identity
coset of $W$. Note that the vertices adjacent to $t$ are of the form $uV$ where $u\in U$ and edges incident on $t$ are $uW$, $u\in U$.
The following lemma is necessary to complete the  proof. 

We refer the reader to \cite{farb-relhyp}
for basics on the construction of electric spaces by
coning off subsets, and to
 \cite{bowditch-relhyp-publish} and \cite{dahmani-mj-height}
for related notions in the context of relative hyperbolicity.

\begin{lemma}\label{ray ct lemma} $ $\\
(1) $X_H$ is hyperbolic relative to the $1$-neighborhoods of the vertex
spaces over the $V$-type vertices. \\ 
(2) Each of the $U$-type vertex spaces is hyperbolic relative to the
corresponding edge spaces. 
In particular, $X_1$ is hyperbolic relative to $\{uX_{12}:u \in U\}$.\\
(3) 
Suppose $\widehat{X}_1$ denotes the graph obtained from $X_1$ by coning off
$\{uX_{12}\}$, $u\in U$. Suppose $\widehat{X}_H$ denotes the coned off graph
obtained by coning the various $1$-neighborhoods of the $V$-type vertex
spaces. Then the natural inclusion $\widehat{X}_1\map \widehat{X}_H$
is a quasi-isometric embedding.
\end{lemma}
\begin{proof}
We note that
$V$ is malnormal and quasiconvex in $H$ by Lemma \ref{mat-oguni elaborate}.
Hence, $H$ is hyperbolic  relative to the cosets of $V$ by \cite[Theorem 7.11]{bowditch-relhyp-publish} (see also 
\cite[Proposition 2.10]{mahan-relrig}). It then follows from Lemma \ref{trees of spaces lemma} that $X_H$ is hyperbolic relative to the 
$V$-type vertex spaces since the natural quasi-isometry between $\Gamma_H$ and $X_H$ takes $\Gamma_V$ to the $V$-type vertices. This implies that $X_H$ is hyperbolic relative to the $1$-neighborhoods of the $V$-type vertices (see \cite[Lemma 4.5, Proposition 4.6]{farb-relhyp} and \cite[Section 3.1]{mahan-ibdd} for instance).
This proves (1). Statement (2) follows from the hypothesis that $W$ is malnormal and quasiconvex in $U$ by applying \cite[Theorem 7.11]{bowditch-relhyp-publish}, or \cite[Proposition 2.10]{mahan-relrig}. 
Statement (3) is evident from
the tree of spaces structure of $X_H$.
\end{proof}

{\em Continuation of the proof of Theorem \ref{prop-suffland}}:
We note that $\gamma'$ is of infinite diameter in $\widehat{X}_H$ by \cite[Theorem 2 (3)]{pranab-ravi-ct} and  converges
to a limit point of $\widehat{X}_1$ in $\pa \widehat{X}_H$. Then 
by \cite[Corollary 2.4]{kap-rafi} (see also \cite[Lemma 3.9]{pranab-ravi-ct})
there is a (uniform) quasigeodesic ray $\sigma$, say, in $\widehat{X}_H$ such that the Hausdorff
distance between $\sigma$ and $\gamma'$ in $\widehat{X}_H$ is finite.
Since the inclusion $\widehat{X}_1\to\widehat{X}_H$ is a qi embedding by Lemma \ref{ray ct lemma} $(3)$, without loss generality we may assume that $\sigma$ is a quasigeodesic in $\widehat{X}_1$ and that it is without back tracking.
Now, de-electrifying $\sigma$ in
$X_H$ gives a uniform quasigeodesic $\gamma''$, say, in $X_H$ asymptotic to
$\gamma'$ (see \cite[Lemma 2.15]{dahmani-mj-height}). Thus it will be enough to show that $\gamma''$ converges to a point of $\pa X_G$.

De-electrifying $\sigma$ in $X_1$ gives a uniform quasigeodesic $\beta$, 
say, in $X_1$ (see \cite[Lemma 2.15]{dahmani-mj-height}). As $(U,G)$ is a ray Cannon--Thurston pair, $\beta$ converges to a point $\xi'$, say, 
of $\pa X_G$. As both $\beta$, $\gamma''$ fellow travel  $\sigma$
outside the subspaces being coned off (see \cite[Lemma 2.13]{dahmani-mj-height}),
any unbounded sequence of points on $\gamma''\cap X_1$ converges to $\xi'$.
 Finally, if $s_n=u_nV$ is the sequence of vertices visited by
$\pi_H\circ \gamma''$ then clearly $d(1, u_nX_{12})\map \infty$ in $X_H$ 
as $n\map \infty$. Thus $d(1,\pi^{-1}_G(s_n))\map \infty$ in $X_G$. Hence, by Lemma \ref{set conv}, the sets $\pi^{-1}_G(s_n)$, which are uniformly quasiconvex in $X_G$, converge to a unique point of $\pa X_G$. It follows that  the connected components of $\gamma''$ contained in the  $V$-type vertex spaces converge to $\xi'$. 
\end{proof}

 As a consequence of Theorem \ref{prop-suffland}, we have the following.
 
 \begin{cor}\label{cor-suffland}
 	Let $G=A\ast_C B$ be a hyperbolic group, and $H=U\ast V$ be a hyperbolic subgroup  such that $A,B,C,U,V,W$ are hyperbolic, and the following hold.
 	\begin{enumerate}
    \item $U<A$, $V<B$,
 		\item $U \cap C = V\cap C =\{1\}$,
 		\item  $C$ is malnormal quasiconvex in $A$,
 		\item The pairs $(U,G)$ and $(V,G)$ admit ray Cannon--Thurston maps.
 	\end{enumerate}
 	Then $(H,G)$ admits a ray Cannon--Thurston map.
 \end{cor}
 
 \begin{proof}
 First we note that the given free product 
 decomposition $H$  gives an immersed subgraph of subgroups of $G=A\ast_C B$: this follows from the hypothesis that $U \cap C = V\cap C =\{1\}$.
 Therefore, the corollary follows  from Theorem \ref{prop-suffland}.
 \end{proof}

\noindent{\bf Domain of continuity}\\
A closer look at the proof of the above theorem shows that we can say more about the ray Cannon--Thurston map obtained in Theorem \ref{prop-suffland}.

\begin{prop}\label{domain of cont prop}
Suppose $H<G$ are as in Theorem \ref{prop-suffland}. 
Suppose, moreover,  $(U,G), (V,G)$ are Cannon--Thurston pairs. Suppose $\gamma$ is
a quasigeodesic ray in $X_H$. Let $\gamma'$ be a quasigeodesic asymptotic to $\gamma$
as in Step 2 of the proof of Theorem \ref{prop-suffland}.
Then the ray Cannon--Thurston map is continuous at
$\gamma(\infty)$ in the following two cases:\\
$(1)$ Image of $\pi_H\circ\gamma'$ has infinite diameter, i.e. $\gamma'$ is of type $(3)$. \\
$(2)$ Image of $\pi_H\circ \gamma'$ is contained in the $1$-neighborhood of a vertex $t$
of $U$-type and no infinite subsegment of $\gamma'$ is contained 
in any $V$-type vertex space.
\end{prop}
\begin{proof}
 Let $i:X_H\map X_G$ denote the inclusion map and $\pa i_r: \pa X_H \map \pa X_G$
 denote a ray Cannon--Thurston map. Suppose $\gamma$ is a geodesic ray in $X_H$
 as in the statement of Proposition \ref{domain of cont prop}, and $\xi=\gamma(\infty)$.
 
 {\bf Claim:} {\em Suppose $\{x_n\}$ is any sequence of points of $X_H$
 converging to $\xi \in \pa X_H $. Then $\{x_n\}$ converges to $\pa i_r(\xi)$ in $X_G$. }
 
 We note that by Lemma \ref{ray ct cont criteria}, and the claim above, it follows that $\pa i_r$ is continuous at $\xi$. 

 {\em Proof of claim.}  Let $\gamma'$ be a quasigeodesic asymptotic to $\gamma$
as in Step 2 of the proof of Theorem \ref{prop-suffland}.

 {\bf Case 1.} Suppose $\pi_H(\gamma')$ has infinite diameter. 
 The proof of the claim in this case follows by the same arguments as in 
 Case $1$ of Step $3$
 of the proof of Theorem \ref{prop-suffland}: one needs to observe that
 for all $n\in \N$ there exists $N\in \N$ such that $x_m\in X_n$ for all $m\geq N$.

\noindent {\bf Case 2.}
Suppose $\gamma'$ is as in (2) of Proposition \ref{domain of cont prop}.
	Since $(U,G)$ is a Cannon--Thurston pair, if a subsequence of $\{x_n\}$ is contained
	in $X_1$ then it converges to $\pa i_r(\xi)$ as the tail of $\gm'$ is not contained in any $V$-type vertex space. Otherwise, assume that there is a subsequence of $\{x_n\}$ whose projections under $\pi_H$ are not $t$. To avoid clumsy notation, let us  denote the subsequence again by $\{x_n\}$. We will show that $\{x_n\}$ converges to $\pa i_r(\xi)$. Choose an unbounded sequence $\{r_n\}\sse\N$ such that $y_n=\gm'(r_n)\in X_1$. Let $v_n$ be the projection of $\pi_H(x_n)$ on $B(t;1)$. We note that $v_n$ is a $V$-type vertex of $T_H$. Let $x'_n\in X_1$ be the point on an $X_H$-geodesic $[y_n,x_n]$ for which $x'_n$ is the farthest point from $y_n$ (with respect to the metric on $X_H$) and suppose that an $X_H$-geodesic $[y_n,x'_n]$ is contained in $X_1$. Note that $\{y_n\}$, $\{x'_n\}$ and $\{x_n\}$ all converge to $\xi$ in $\pa X_H$. As $(U,G)$ is a Cannon--Thurston pair, $\{y_n\}$ and $\{x'_n\}$ converge to $\pa i_r(\xi)$ in $\pa X_G$. 
	 On the other hand,
	if $A_n$ is the connected component of $T_G\setminus \{t\}$ containing $v_n$
	then $\pi^{-1}_G(A_n)=Q_n$, say, is a uniformly quasiconvex set in $X_G$.  Fix  $x_0\in X_1$. Then,  
	$d_{X_G}(x_0, Q_n)\map \infty$ as $n\map \infty$ as $\xi$ is not
	a  limit point of any $V$-type vertex space in $X_H$. In particular, $d_{X_G}(x_0,[x'_n,x_n]_{X_G})\to\infty$ as $n\to\infty$ since $x_n\in Q_n$
	and $x'_n\in N_1(Q_n)$. Finally, as $\{x'_n\}$ converges in $\pa X_G$, $\{x_n\}$ also converges in $\pa X_G$ to the same limit as $\{x'_n\}$. This completes the proof.
\end{proof}
We conclude this section with the following remark, which explains that all the results in this section remain valid when the assumption of malnormality is replaced by that of almost malnormality.
\begin{remark}\label{rmk-malnormal to almal}
In the proof of Theorem \ref{prop-suffland}, the crucial place where the malnormality of $W$ in $U$ is used is in the proof of Lemma \ref{ray ct lemma}. There we have applied Lemma \ref{mat-oguni elaborate} to conclude that $V$ is (almost) malnormal in $H$. This is needed to apply \cite[Theorem $7.11$]{bowditch-relhyp-publish} to conclude that $X_H$ is hyperbolic relative to cosets corresponding to $V$. However, Bowditch's theorem (and Lemma \ref{mat-oguni elaborate}) are also valid when the corresponding subgroup is almost malnormal. Consequently, Theorem \ref{prop-suffland} continues to hold if $W$ is almost malnormal in $U$. 

Secondly, Proposition \ref{domain of cont prop} relies on Theorem \ref{prop-suffland} and on the fact that $B$ is quasiconvex in $G$. The latter fact is obtained by applying Lemma \ref{mat-oguni elaborate} to $G=A\ast_C B$. 
Therefore, Proposition \ref{domain of cont prop} remains valid if $C$ and $W$ 
are almost malnormal.
\end{remark}

\section{JKLO criterion and non-existence of Cannon--Thurston Maps}
\label{sec-crit}
 The term JKLO criterion below is named after Jeon, Kapovich, Leininger and Ohshika. The following can be derived from  their theorem \cite[Theorem A(2)]{JKLO}. We include a proof for  completeness.

 \begin{prop}\label{prop-jklo} Suppose that a hyperbolic group $G$  acts properly by isometries on a proper hyperbolic
 	metric space $X$. Let $\rho: G\map X$ be an orbit map $g\mapsto g.x$ where $x\in X$.
 	Suppose that there exist two sequences $\{g_n\},\{g'_n\}$ in $G$ such that
 	\begin{enumerate}
 	\item $\{g_n\},\{g'_n\}$ converge to two
 	distinct points of $\pa G$,
 	\item $\LMX g_nx=\LMX g'_n x=\xi\in\pa X$,
 	\item $\xi$ is a conical limit point of $G$.
 	\end{enumerate} 
 	 Then $\rho$ does not admit a Cannon--Thurston map.
 
 	In particular, the following is true: 
    Suppose $G<G'$ is a Cannon--Thurston pair of hyperbolic groups. Let $\pa i:\pa G\to\pa G'$ be the Cannon--Thurston map.
  Then for any conical limit
 	point $\xi\in \partial G'$ of $G$, $(\partial i)^{-1}(\xi)$
 	is a singleton set. 	
 \end{prop}
 
 \begin{proof} The second assertion clearly follows from the first; so we prove the first one only.
 	
 	Suppose $X$ is $\delta$-hyperbolic.
 	Suppose there is a Cannon--Thurston map $\partial \rho:\partial G\map \partial X$.
 	Let $\al$ be a geodesic ray in $X$ joining $x$ to $\xi$. Since $\xi$ is a conical limit point of $G$, there is an unbounded sequence $\{h_n\}$ in $G$ such that $d_X(h_nx,\al)\le R$ for some
 	$R>0$. After passing to a subsequence, if necessary, we assume that
 	$\lim^G_{n\ri\infty} h_n=\eta\in \partial G$. Then it follows that $\LMX h_nx=\pa \rho(\eta)=\xi$.
 	
 	
Since $\{g_n\}$ and $\{g'_n\}$ converge to two distinct points of $\pa G$, without loss of generality, we assume that $\lim^G_{n\to\infty}h_n\ne\lim^G_{n\to\infty}g_n$. 	As $\lim^X_{n\map\infty} g_nx=\xi$, if $\al_n$ is a geodesic joining $x$ and
 	$g_nx$ for all $n\in \N$ then (up to passing to a subsequence) $\al_n$'s converge
 	(uniformly on compact subsets) to a geodesic ray $\bt$, say,
 	joining $x$ to $\xi$. By Lemma \ref{basic bdry lemma}, $Hd(\al,\bt)\leq 2\delta$.
 	Hence, there is a subsequence of natural numbers $\{n_i\}$, and a
 	sequence of real numbers $\{r_i\}$ such that
 	$d_X(h_ix,\al_{n_i}(r_i))\le R+2\delta+1$ for all $i\in\N$.
 	As the $G$-action is by isometries on $X$, it follows that
 	the geodesic $h^{-1}_i\al_{n_i}$ in $X$ joining $h_i^{-1}x$ and $h_i^{-1}g_{n_i}x$
 	passes through the $(R+2\delta+1)$-radius ball centered at $x$.

 	On the other hand, as the sequences $\{h_n\}$
 	and $\{g_n\}$ are converging to different points in $\pa G$, it follows that
 	$d_G(h_m, [1,g_n])\map \infty$ as $m,n\map \infty$.
 	Hence, $d_G(h_i, [1,g_{n_i}])\map \infty$ as $i \map \infty$. Therefore,
 	$d_G(1, [h^{-1}_i,h^{-1}_ig_{n_i}])\map \infty$ as $n\map \infty$.
 	
 	This violates Lemma~\ref{mitra's criterion} contradicting  the assumption that
 	$\rho:G\map X$ admits a Cannon--Thurston map.
 \end{proof}
 
 A pair $(G,G')$ as in Proposition~\ref{prop-jklo} will be said to satisfy the  \emph{JKLO criterion}.
The aim of the authors in \cite{JKLO} was to furnish a property of conical limit points in the presence of a Cannon--Thurston map. Proposition~\ref{prop-jklo} turns this around into a criterion for the non-existence
of Cannon--Thurston maps.

As a direct consequence of the JKLO criterion we describe the following
general group theoretic situation where there are non Cannon--Thurston pairs. Instead of directly applying the JKLO criterion, this proposition is used later on to produce examples of non Cannon--Thurston pairs of hyperbolic groups.
\begin{prop}\label{prop-gpcriterion1}{\em (A group theoretic criterion)}
	Suppose $G_1<G$ are hyperbolic groups and $K, P$ are two infinite hyperbolic subgroups of $G_1$
	such that the following hold:
	\begin{enumerate}
		\item $P$ is quasiconvex in $G$. 
		\item $K$ is quasiconvex in $G_1$.
		\item $P\cap K$ is finite.
		\item There is a sequence $\{x_n\}$ in $K$ such that $\lim^K_{n\map \infty} x_n\in \partial K$
		and $\lim^G_{n\map \infty} x_n\in \Lambda_G(P)$.
	\end{enumerate}
	Then $(G_1, G)$ is not a Cannon--Thurston pair.
\end{prop}
\proof First, we note that $P$ is quasiconvex in $G_1$ since it is so in the larger group $G$.
Let $\xi=\lim^G_{n\map \infty} x_n$. As $P$ is quasiconvex in $G$, $\xi$ is a conical limit point of $P$ and hence it is a conical
limit point of $G_1$. Let $\gamma_1$ be a geodesic ray in $P$ starting from $1$ such that
$\partial i_{P,G}(\gamma_1(\infty))=\xi$. Let $\gamma_2$ be a geodesic ray in $K$ starting from
$1$ such that $\gamma_2(\infty)=\lim^K_{n\map \infty} x_n$. As $P, K$ are both
quasiconvex in $G_1$, so $\gamma_1, \gamma_2$ are both quasigeodesics in $G_1$. Since $P\cap K$ is finite, $\Lambda_{G_1} P \cap \Lambda_{G_1} K = \Lambda_{G_1} (P \cap K) = \emptyset$ \cite{short}.
Hence, $\gamma_1, \gamma_2$ converge to two different points of $\partial G_1$.
If $(G_1,G)$ is a Cannon--Thurston pair, then $\gamma_1(\infty), \gamma_2(\infty)\in \partial G_1$ would
be two distinct points in  $\partial i_{G_1,G}^{-1}(\xi)$.  This violates  the JKLO criterion (Proposition~\ref{prop-jklo}).
\qed

\section{Main technical theorem}\label{sec-main thm}

In this section we prove the main technical theorem of the paper.

\begin{theorem}\label{cor-gpcriterion1}{\em (Main technical theorem)}\label{main thm}
Let $H$ be a hyperbolic group. Let $F, K$ be two hyperbolic subgroups of $H$.
	Suppose, moreover, that the following hold.
\begin{enumerate}
\item $F$ is malnormal and quasiconvex in $H$.
\item $F\cap K=\{1\}$.

Further, suppose that $\phi:F\to F$ is a hyperbolic endomorphism of $F$. Let $G$ be the HNN extension $G=H\ast_{\phi}$ with stable letter $t$; and let
    $G_1$ be the subgroup of $G$ generated by $K\cup \{t\}$.

\item There is a sequence $\{y_n\}$ in $K$ converging to a point of $\pa K$ such that
		$\lim^G_{n\map\infty} y_n=\lim^G_{n\map\infty} t^n\in\pa G$.
	\end{enumerate}
    
	Then
	\begin{enumerate}
	\item $G_1, G$ are hyperbolic,
	\item $(G_1, G)$ is not a Cannon--Thurston pair.
	\item If a ray Cannon--Thurston map exists for the pair $(K,H)$ then we have a ray Cannon--Thurston map $\pa i_{r,G_1,G}:\pa G_1\to\pa G$.

\item The ray Cannon--Thurston map $\pa i_{r,G_1,G}:\pa G_1\to\pa G$ $($in $(3))$ is not continuous at the points in the $G_1$-orbit of $\lim^{G_1}_{n\to\infty}t^{-n}$  in $\pa G_1$.
    
    \item If, moreover, $(K,H)$ is a Cannon--Thurston pair, then a ray Cannon--Thurston map 
for the pair $(G_1, G)$ is continuous on the complement of the $G_1$-orbits in $\partial G_1$ of the points $\lim^{G_1}_{n\map \infty}t^{\pm n}$.
\end{enumerate} 
\end{theorem}
\begin{proof}
$(1)$ 
 Since $F\cap K=\{1\}$, it follows from Britton's lemma that $G_1=K*\langle t\rangle$. Hyperbolicity of $G_1$ is immediate from this. Let $L$ be the HNN extension
 $L=F\ast_{\phi}$. Then by definition $L$ is hyperbolic. On the other hand,
 we note that $G$ is isomorphic to the amalgamated free product
 $H\ast_F L$. Then the hyperbolicity of $G$  follows from  Lemma \ref{mat-oguni elaborate}.

\smallskip
$(2)$ That a Cannon--Thurston map does not exist for the pair $(G,G_1)$ follows from Proposition \ref{prop-gpcriterion1}.
Indeed, since $K$ and $\langle t\rangle=P$, say, are retracts of $G_1$, they
are both quasiconvex in $G_1$. Since $P$ is cyclic, it is
 quasiconvex in $G$ as well. Thus hypotheses (1), (2) of
Proposition \ref{prop-gpcriterion1} are satisfied. 
Hypothesis (3) of Proposition \ref{prop-gpcriterion1} is immediate, since $P\cap K=\{1\}$.
Hypothesis (4) of Proposition \ref{prop-gpcriterion1}
is an immediate consequence of hypothesis (3) of the present Theorem. Hence, by Proposition
\ref{prop-gpcriterion1} $(G_1,G)$ is  not a Cannon--Thurston pair.

\smallskip
$(3)$ This part of the theorem is deduced from Corollary \ref{cor-suffland} as follows.
As in the proof of $(1)$, we have $G=H\ast_F L$ and $G_1=K\ast P$. 
The only hypothesis of Corollary \ref{cor-suffland} that needs checking is that $(K,G)$ is a ray Cannon--Thurston pair.  By \cite[Corollary 3.11]{mitra-trees},  $(H,G)$ is a Cannon--Thurston pair. (Note that $G=H*_{\phi}$, and $\phi(F)$ is quasiconvex in $F$ by definition and hence quasiconvex in $H$.) Given that $(K,H)$ is a ray Cannon--Thurston pair, it immediately follows that $(K,G)$ is ray Cannon--Thurston pair.

\smallskip

$(4)$ As $K$ is qi embedded in $G_1$, $\{y_n\}$ converges to
a point $\eta\in \pa G_1$, say. We note that 
$\eta \neq \lim^{G_1}_{n\map \infty}t^{\pm n}
\in \pa G_1$. 
One way to see this is to observe that both $K$ and $\langle t\rangle$ are qi embedded
in $G_1$ and $K\cap \langle t\rangle=\{1\}$ by hypothesis. As $\eta\in \Lambda_{G_1}(K)$
we are done by the standard limit set intersection property of
quasiconvex subgroups of hyperbolic groups (\cite[Proposition 3]{short}).

Let $\eta_k=t^{-k}\eta$ for all $k\in \N$. 
Then we note that 
$\lim^{G_1}_{k\map \infty} \eta_k= \lim^{G_1}_{k\map \infty} t^{-k} \eta =
\lim^{G_1}_{n\map \infty}t^{-n}$, e.g. by \cite[8.1.G]{gromov-hypgps},
as $\eta\neq \lim^{G_1}_{n\map \infty}t^{\pm n}$.
But $\pa i_{r,G_1,G}(\eta_k)= \pa i_{r,G_1,G}(t^{-k}\eta)=
t^{-k}\pa i_{r,G_1,G}(\eta)=t^{-k} \lim^G_{n\map \infty} t^n=\lim^G_{n\map \infty} t^n$. 
As 
$\pa i_{r,G_1,G}(\lim^{G_1}_{n\map \infty}t^{\pm n})=\lim^{G}_{n\map \infty}t^{\pm n}$, and $\lim^{G}_{n\map \infty}t^{ n}\neq \lim^{G}_{n\map \infty}t^{- n}$,
it follows that $\pa i_{r,G_1,G}$ is not continuous at $\lim^{G}_{n\map \infty}t^{- n} \in \pa G_1$.\smallskip

$(5)$ Let $\pi_{G_1}:X_{G_1}\map T_{G_1}$
and $\pi_G:X_G\map T_G$ be the trees of graphs associated to 
the amalgamated free product decompositions $G=H\ast_F L$ and $G_1=K\ast P$
respectively (cf.\ Step 1 of the proof of Theorem
\ref{prop-suffland}). We note that in $X_{G_1}$ every vertex space is isometrically
embedded. Any geodesic ray $\gamma\subset X_{G_1}$, 
(after
possibly removing a finite subsegment)   satisfies one of the following conditions: \\
$(i)$ $\pi_{G_1}(\gamma)$ has infinite diameter.\\
$(ii)$ $\pi_{G_1}(\gamma)$ is a $K$-type point of $T_{G_1}$.\\
$(iii)$ $\pi_{G_1}(\gamma)$ is a $\langle t\rangle$-type point of $T_{G_1}$.

By Proposition \ref{domain of cont prop} if $\xi\in \pa X_{G_1}$ is
the limit point of a quasigeodesic ray of type $(i)$ or $(ii)$ then the
ray Cannon--Thurston map is continuous at $\xi$.  The collection of such points is precisely
the complement of the $G_1$-orbits of $\lim^{G_1}_{n\map \infty}t^{ \pm n}$.
This completes the proof of part (5) of Theorem \ref{cor-gpcriterion1}.
\end{proof}

The theorem above ensures  discontinuity  at  points in the $G_1$-orbit of $t^{-\infty}$. It also ensures continuity in the complement of the 
$G_1$-orbits of $t^{\pm\infty}$. However, we are unable to make such a definitive statement about 
 points in the $G_1$-orbit of $t^{\infty}$.
We shall see in the following section examples where Theorem \ref{cor-gpcriterion1}
will apply and where $\pa i_{r,G_1,G}$ will be discontinuous on the $G_1$-orbits
of  $\lim^{G_1}_{n\map \infty}t^{\pm n}$  in $\pa G_1$.

We comment on how to obtain all the conclusions of Theorem \ref{main thm} under the assumption that $F\cap K$ is finite and $F$ is almost malnormal in $H$.
\begin{remark}
Hyperbolicity of $G$ in Theorem \ref{cor-gpcriterion1} would follow by the same argument if we
    assume $F$ is almost malnormal instead of
    malnormal. Also we could
    replace the assumption $F\cap K=\{1\}$ by requiring that $F\cap K$ is finite and that the endomorphism $\phi$ restricted to $F\cap K$ is the identity map. In that case, we note that if $A=F\cap K$ then $\langle K,t\rangle=K\ast_A (\langle t\rangle\times A)$. Note that Theorem \ref{prop-suffland} and Proposition \ref{domain of cont prop} remain valid when the malnormality assumptions are replaced by almost malnormality (see Remark \ref{rmk-malnormal to almal}). Therefore, under the assumption that $F$ is almost malnormal and $F\cap K$ is finite, the only part that requires checking is the proof of part $(3)$ of Theorem \ref{cor-gpcriterion1}. In this case, one directly applies Theorem \ref{prop-suffland} instead of Corollary \ref{cor-suffland}
    to obtain the necessary conclusion. 
\end{remark}

\section{Applications}\label{sec-applns}
By Theorem \ref{cor-gpcriterion1} if we are able to find groups
$H, K,F$ and a hyperbolic endomorphism $\phi$ of $F$ satisfying the
hypotheses of Theorem \ref{cor-gpcriterion1} then we would have
a non-Cannon--Thurston pair. In this section, we exhibit three
classes of examples.

\subsection{Extensions of surface groups and free groups}\label{sec-normal}
\begin{theorem}\label{extn case}
Suppose $H$ is a hyperbolic group with hyperbolic subgroups $F, K$
such that the following hold:
\begin{enumerate}
\item $F$ is a malnormal quasiconvex free subgroup of $H$ of rank at least $3$.
\item $(K,H)$ is a Cannon--Thurston pair and $\Lambda_H(K)=\pa H$.
\item $F\cap K=(1)$.
\end{enumerate}

For  a hyperbolic automorphism $\phi:F\map F$, let
$L= F\rtimes_{\phi} \Z$, $G=H\ast_F L$ and $G_1=\langle K,t\rangle$
where $t$ is the generator of $\Z<L$ corresponding to $\phi$.
Then we have the following. 
\begin{enumerate}
    \item $G, G_1$ are hyperbolic, $G_1\simeq K \ast \langle t\rangle$.
    \item  $(G_1,G)$ is only a ray Cannon--Thurston pair.

    \item Lastly, $\pa i_r:\pa G_1 \map \pa G$ is discontinuous
    on the $G_1$-orbits of $\lim^{G_1}_{n\map \infty}t^{\pm n}$
    and continuous everywhere else on $\pa G_1$.
\end{enumerate}
\end{theorem}

\begin{proof} $(1)$, $(2):$
By Theorem \ref{cor-gpcriterion1}, we
only need to show the existence of a sequence $\{y_n\}$ in $K$
such that  $\lim^G_{n\map \infty} y_n= \lim^G_{n\map \infty} t^n$.
As $F$ is
a normal infinite subgroup of $L$, $\Lambda_L(F)=\pa L$. In particular,
there is a sequence $\{x_n\}$ (resp.\ a sequence $\{x'_n\}$)
  in $F$ such that $\lim^F_{n\map \infty} x_n\in \pa F$ (resp.
 $\lim^F_{n\map \infty} x'_n\in \pa F$) such that
$\lim^L_{n\map \infty} x_n=\lim^L_{n\map \infty} t^n$ 
(resp.\ $\lim^L_{n\map \infty} x'_n=\lim^L_{n\map \infty} t^{-n}$). 
As $L$ is qi embedded in $G$ by Lemma \ref{mat-oguni elaborate}, it follows that $\lim^G_{n\map \infty} x_n=\lim^G_{n\map \infty} t^n$
 (resp.\ $\lim^G_{n\map \infty} x'_n=\lim^G_{n\map \infty} t^{-n}$). 
 On the other hand, as $\Lambda_H(K)=\pa H$, we can find a 
sequence $\{y_n\}$ (resp. $\{y'_n\}$) in $K$ such that 
$\lim^H_{n\map \infty} y_n= \lim^H_{n\map \infty} x_n$
(resp.\ $\lim^H_{n\map \infty} y'_n= \lim^H_{n\map \infty} x'_n$).
Again as $G=H\ast_{\phi}$, where $H$ is hyperbolic and $F$ is quasiconvex
(and malnormal) in $H$,  it follows from \cite[Corollary 3.11]{mitra-trees}
that $(H,G)$ is Cannon--Thurston pair. Therefore,
$\lim^G_{n\map \infty} y_n= \lim^G_{n\map \infty} x_n=\lim^G_{n\map \infty} t^n$
(resp.\ $\lim^G_{n\map \infty} y'_n= \lim^G_{n\map \infty} x'_n=\lim^G_{n\map \infty} t^{-n}$).

$(3):$ By Theorem \ref{cor-gpcriterion1} $(5)$ the map
$\pa i_r$ is continuous on the complement of $G_1$-orbits of the points
$\lim^{G_1}_{n\map \infty}t^{\pm n}$, whereas by Theorem \ref{cor-gpcriterion1} $(4)$, $\pa i_r$ is discontinuous
on the $G_1$-orbit of $\lim^{G_1}_{n\map \infty}t^{- n}$.
However, replacing the sequence $\{y_n\}$ by $\{y'_n\}$ as defined above
and $\{t^n\}$ by $\{t^{-n}\}$ in Theorem \ref{cor-gpcriterion1} $(4)$
we see that $\pa i_r$ is discontinuous on the $G_1$-orbit of
$\lim^{G_1}_{n\map \infty}t^{n}$ as well. This completes the proof of $(3)$
and with it that of Theorem \ref{extn case}.
\end{proof}

    As a consequence of the above theorem we have the following.
    
\begin{theorem}\label{main application}
Suppose $H = K \rtimes \mathbb F$ where 
\begin{enumerate}
    \item Either $K=\mathbb F_n$ where $n\geq 3$ or $K=\pi_1(S)$
    where $S$ is a closed, orientable surface of genus at least 2.
    \item $H$ is hyperbolic.
    \item $\mathbb F=\mathbb F_m$, $m\geq 3$.
\end{enumerate}

Let $\phi: \mathbb F \map \mathbb F$ be a hyperbolic automorphism. Let $L=\mathbb F *_\phi$ and let $t$ denote the stable letter in $L$. Finally, let $G= H*_{\mathbb F} L$ and 
$G_1 = \langle K,t\rangle$. Then the following hold:
 
 \begin{enumerate}
 	\item $G$ is hyperbolic
    \item $G_1\simeq K\ast \langle t\rangle$ whence it is hyperbolic.
 	\item  The pair $(G_1,G)$ admits only a ray Cannon--Thurston map. 
    \item  $\pa i_r:\pa G_1 \map \pa G$ is discontinuous
    on the $G_1$-orbits of $\lim^{G_1}_{n\map \infty}t^{\pm n}$ and continuous
    everywhere else on $\pa G_1$.
 \end{enumerate}
\end{theorem}

\begin{proof}
We need to only verify conditions (1) and (2) of Theorem \ref{extn case}  
(conditions (3) and (4) are clear). Malnormality of $\mathbb F_m$ in $H$
follows from Lemma \ref{malnormal basic}. Quasiconvexity of $\mathbb F_m$ 
in $H$ is deduced as follows. We note that there is a natural group homomorphism
    $\pi: H\map \mathbb F_m$ such that $\pi|_{\mathbb F_m}$ is the identity
    map. It then follows that $\mathbb F_m$ is qi embedded in $H$, and
    hence quasiconvex in $H$. Thus (1) is verified. For (2), we note that
    $\Lambda_H(K)=\pa H$ as $K$ is an infinite normal subgroup of $H$. 
    Lastly, the existence of Cannon--Thurston map $\pa K \map \pa H$ follows from 
    \cite[Corollary 3.11]{mitra-ct}. 
\end{proof}
We note that for $K=\pi_1(S)$, existence of such a hyperbolic $H$ was proven by Mosher \cite{mosher-hbh}. For $K=\mathbb{F}_n$, $n\geq 3$ existence of such a hyperbolic $H$ was proven by Bestvina--Feighn--Handel \cite{BFH-lam}.

\begin{remark}\label{normal remk}
In Theorem \ref{main application}, assumption (2) can be replaced by
requiring that $K$ is merely a normal subgroup (i.e. perhaps not a free subgroup or
not isomorphic to a surface subgroup) of $H$ keeping the rest of the
hypotheses intact. 

More generally, one could take a general torsion-free $H$ where $K$ is a non-elementary hyperbolic normal subgroup  (whence it would follow that $H/K$ is a hyperbolic group
by \cite[Theorem A]{mosher-hypextns}). Assume also that $H/K$ is also non-elementary. 
Then, by a standard ping-pong argument and Theorem \ref{main malnormal thm}, one can construct a free, malnormal, quasiconvex subgroup of $H/K$ of rank at least $3$. Finally, Lemma \ref{malnormal basic} yields a free, malnormal, quasiconvex subgroup $\mathbb F$ of $H$. Then one applies Theorem \ref{cor-gpcriterion1}.
\end{remark}

\subsection{Constructions using commensurated subgroups}\label{sec-comm}
In this subsection we show another way in which one may construct triples
$(H,K,F)$ of hyperbolic groups satisfying hypotheses (1), (2), (3)
of Theorem \ref{extn case}. Since the proofs are standard we give
a sketch only. For  examples of commensurated hyperbolic subgroups $K$ of hyperbolic groups $H$ for which $K$ need not be normal in $H$, we refer the reader to \cite{min,ghosh-mj}. We recall that the {\em Cayley-Abels graph} $\mathcal G_{H/K}$ of the pair
$(H,K)$ with respect to a finite generating set $S$ of $H$ is the simplicial graph with vertex set $H/K$ and edge set
$\{(gK,gsK)| g\in H, s \in S\}$.
\begin{theorem}\label{main thm for comm}
	Suppose $K<H$ are non-elementary torsion-free hyperbolic groups such that
	$H=Comm_H(K)$,  and the (Gromov) boundary of the
	Cayley--Abels graph $\G_{H/K}$ contains at least $3$ points.
	Then the following hold:
    \begin{enumerate}
    \item 	$(K,H)$ is a Cannon--Thurston pair and $\Lambda_H(K)=\pa H$.
    \item   There is a malnormal quasiconvex free subgroup $F$ 
    of rank at least $3$ of $H$ such that $K\cap F=\{1\}$.

    Consequently, if 
	$\phi:F\map F$ is a hyperbolic automorphism then the following hold.
	\item 	The HNN extension $G=H*_{\phi}$ is hyperbolic.
	
	\item  The subgroup $G_1$ of $G$ generated by $K$ and the stable letter $t$ of
	the HNN extension is hyperbolic.
	
	\item  $(G_1,G)$ is only a ray Cannon--Thurston pair.

\item Finally, the ray Cannon--Thurston map $\pa i_r:\pa G_1\to\pa G$ is discontinuous on the $G_1$-orbits of $\lim^{G_1}_{n\to\infty}t^{\pm n}$ and continuous everywhere else on $\pa G_1$.
	
\end{enumerate}
\end{theorem}
{\em Sketch of proof:} 
 As $K$ is infinite,
it is a standard fact that $\Lambda_H(K)=\pa H$. It follows from
\cite[Proposition 5.12]{NirMaMj-commen} and \cite[Theorem 5.3]{pranab-mahan} that a Cannon--Thurston map
$\pa K \map \pa H$ exists. This completes checking hypothesis (2) of Theorem
\ref{extn case}.

The Cayley--Abels graph $\G_{H/K}$ for the pair $(H,K)$ 
 is a hyperbolic space by 
\cite[Corollary 5.13]{NirMaMj-commen}.  As $\pa \G_{H/K}$ has at least
three points by hypothesis, $\G_{H/K}$ is not 
quasi-isometric to a line and also, by \cite[Proposition 1.3 (3)]{NirMaMj-commen},
$H$ does not a fix a point in $\pa \G_{H/K}$.
Then there are elements  $h_1, h_2\in H$ which
act as independent loxodromics on $\G_{H/K}$ by
\cite[Theorem 3, Corollary 4]{woess-Fixed}.  By a
standard ping-pong argument $\langle h^n_1, h^n_2\rangle$ gives a free
quasiconvex subgroup $\mathbb F$ of $H$ for all large $n$.  We note that 
$K\cap F_1=(1)$. One may then construct a  finitely
generated free group $F<F_1$ of rank at least $3$ which is malnormal and quasiconvex in $H$, by
\cite{kapovich-nonqc}. Thus hypothesis (1)
of Theorem \ref{extn case} is verified.
Hypothesis (3) follows from a reprise of  the argument in the last paragraph of Theorem~\ref{extn case}. The conclusion follows.
\qed


\subsection{Examples built from HNN extensions of free groups}\label{sec-endo}
In this section we construct $H$ coming from 
 multiple ascending HNN extension of free groups.

Suppose $K=\langle S|R\rangle$ is any group and $\psi_i:K\map K$, $1\leq i \leq n$ are $n$
injective endomorphisms of $K$. Then the multiple
ascending HNN extension of $K$ with respect to these endomorphisms is the
group 
$$H:=\langle S\sqcup \{t_i:1\leq i\leq n\}|R\sqcup \{t_jxt^{-1}_j \psi_j(x)^{-1}, x\in S, 1\leq j\leq n\}\rangle.$$ We shall refer to the $t_i$'s as the \emph{stable letters} for the 
multiple HNN extension.
However, it is easy to see that $H$ is the fundamental group of a graph
of groups where 
\begin{enumerate}
    \item the underlying graph is a `rose with $n$ petals', i.e.\
it has one vertex $v$,  and $n$ oriented loops $e_1,\cdots, e_n$ based at $v$;
\item the vertex group and the edge groups are all $K$;
\item for any edge $e_i$, the two homomorphisms
$G_{e_i}=K\map G_{t(e_i)}=K$ and $G_{e_i}=K\map G_{o(e_i)}=K$ are
respectively $\psi_i$ and the identity.
\end{enumerate}
The lemma below follows immediately from this.
\begin{lemma}\label{lem-useful}
  Let $F=\langle t_1, \cdots, t_n\rangle <H$. Then  $F\simeq \mathbb F_n$
  and there is a group homomorphism $\psi:H\map F$ such that
  $\psi|_F$ is the identity map on $F$. In particular, 
  $H=F\ltimes Ker(\psi)$. Moreover, $Ker(\psi)$ is the normal
  closure of $K$ in $H$. 
\end{lemma}

Lemma \ref{lem-useful} has the following corollary using Lemma \ref{malnormal basic}.
\begin{cor}\label{endo cor}
    If $K=\mathbb F_m$ and $H$ is hyperbolic then the following hold:\\
    (1) $F$ is a malnormal, quasiconvex subgroups of $H$.
    (2) $K\cap F=(1)$.
\end{cor}

We give below the main theorem of this section demonstrating the role of a multiple
HNN extension.

\begin{theorem}\label{thm-endo} 
Suppose $K=\mathbb F_m$ is the free group on $m$ letters where $m\geq 2$.
Suppose $\psi_i$, $1\leq i\leq n$ are $n$ positive (hyperbolic)
endomorphisms of $K$ where $n\geq 3$ such that the resulting multiple
ascending HNN extension $H$ is hyperbolic.
Let  $t_1, \cdots, t_n$ be the corresponding stable letters and let 
$F=\langle t_1, \cdots, t_n\rangle$.

	Let $\phi:F\map F$ 
    be a positive, hyperbolic endomorphism of $F$. Let $G$ be the HNN extension given by $$G=H*_{F,\phi}=\langle H,t:t\,t_i\,t^{-1}=\phi(t_i),\, i=1\cdots, n\rangle.$$
	
	Then 
	\begin{enumerate}
\item	$G$ and $G_1=\langle K,t\rangle=K*\langle t\rangle$ are hyperbolic.
\item There is a sequence $\{y_n\}$ in $K$ such that
	$\lim^G_{n\map \infty}  y_n=\lim^G_{n\map \infty} t^n$.  

    Consequently, all the conclusions of Theorem \ref{main thm} hold;
    in particular,  $(G_1,G)$ admits only a ray Cannon--Thurston map.
	\end{enumerate}
\end{theorem}

We start with the following clarifying remark. Suppose $\mathbb F$ is a free group
on $k\geq 2$ letters. Suppose we fix a free basis $S=\{x_1, \cdots, x_k\}$ of
$\mathbb F$. Then an element $w\in \mathbb F$ is called {\em positive}
if $w$ is the product of only positive powers of $x_i$'s. Similarly
in the Cayley graph $\Gamma(\mathbb F, S)$ a geodesic ray $\gamma$ will be
called a {\em positive ray} if $\gamma(n)$, $n\in \N$ are all positive words of
$\mathbb F$. Finally an endomorphism $\phi:\mathbb F\map \mathbb F$ will
be called a {\em positive endomorphism} if $\phi(x_i)$, $1\leq i\leq k$
are all positive elements of $\mathbb F$. We emphasize that the notion of {\em positivity}
depends on a choice of free basis, e.g. in the above theorem we fix 
$\{t_1, \cdots, t_n\}$ as the free basis of $F$.
We start with the following observation for the proof of Theorem \ref{thm-endo}.

\begin{lemma}\label{t infinity0}
Let $P\simeq\mathbb F_m =\langle x_1, \cdots, x_m\rangle$
and $\phi_i:P\map P$, $1\leq i\leq n$ be positive hyperbolic endomorphisms
such that the resulting multiple ascending HNN extension $L$, say, is hyperbolic.
Let $s_1, \cdots, s_n$ be the stable letters and let $Q=\langle s_1, \cdots, s_n \rangle$. 
Then \\
(1) $(P,L)$ is a Cannon--Thurston pair.\\
(2) For any positive ray $\alpha$ in $Q$, there is a positive ray
 $\gamma$ in $P$ such that $\pa i_{P,L}(\gamma(\infty))=\lim^L_{n\map \infty} \gamma(n)=\lim^L_{n\map \infty} \alpha(n).$
\end{lemma}
\proof First, note that $(P,L)$ is a Cannon--Thurston pair by 
the main theorem of \cite{mitra-trees}. Secondly, note that compositions
of positive endomorphisms are positive endomorphisms. Now, let $x$ be one of the basis elements of $P$. As $L$ is hyperbolic, $\{\alpha(n)(x)\}$ is a
sequence of positive elements in $P$ whose lengths are going to $\infty$. 
It follows that $d_L(1, \alpha(n)(x))\map \infty$ as $n\to\infty$. On the other hand,
$Q$ is qi embedded in $L$ by Corollary \ref{endo cor} whence $\alpha$ is a 
quasigeodesic in $L$. Hence, by Lemma \ref{limit g^n}
$\lim^L_{n\map \infty} \alpha(n)=\lim^L_{n\map \infty} \alpha(n)(x)$. 
Finally, since the elements $\alpha(n)(x)$
are positive words in $P$,  we can extract a subsequence of 
$\alpha(n)(x)$ which converges to a positive ray $\gamma$ in $P$. 
It then follows that $\pa i_{P,L}(\gamma(\infty))=\lim^L_{n\map \infty} \gamma(n)=\lim^L_{n\map \infty} \alpha(n).$ \qed

\begin{proof}[\bf Proof of Theorem~\ref{thm-endo}:]
 By Corollary \ref{endo cor} $F$ is a malnormal quasiconvex subgroup of $H$ and
$F\cap K=(1)$. (1) follows from this and Corollary \ref{basic combination}.

The proof of (2) runs as follows.
Let $L=F\ast_{\phi}$. It follows that $G=H\ast_F L$ and hence $L$ is
quasiconvex in $G$ by Lemma \ref{mat-oguni elaborate}. Now, by Lemma
\ref{t infinity0} there is a positive geodesic ray $\gamma$, say, in $F$
such that $\lim^L_{n\map \infty} \gamma(n)= \lim^L_{n\map \infty} t^n$.
Hence, 
$${\lim}_{n\map \infty}^{G} \gamma(n)= {\lim}^G_{n\map \infty} t^n\,\, \hfill \ \ \ \ \ \ \ \ (*)$$ as
$L$ is quasiconvex in $G$. Once again, by Lemma \ref{t infinity0} 
applied to $H$ there is a positive geodesic ray $\alpha$, say, in $K$
such that $\lim^H_{n\map \infty} \gamma(n)= \lim^H_{n\map \infty} \alpha(n)$.
This implies that 
$${\lim}^G_{n\map \infty} \gamma(n)= {\lim}^G_{n\map \infty} \alpha(n)\,\, \hfill \ \ \ \ \ \ \ \ (**)$$
as $(H,G)$ is a Cannon--Thurston pair. The latter claim follows from the main
theorem of \cite{mitra-trees} as $G$ is an HNN extension of $H$ where the edge
group $F$ is quasiconvex in $H$. However, from $(*)$ and $(**)$ we have
$\lim^G_{n\map \infty} \alpha(n)= \lim^G_{n\map \infty} t^n$. This completes the
proof of (2).
\end{proof}

\begin{remark}\label{rmk-endo}
In the proof of Theorem~\ref{thm-endo}, the hypothesis that
the stable letters $t_i$'s correspond to positive hyperbolic endomorphisms $\psi_i$'s is used only in the proof of Lemma~\ref{t infinity0}, to conclude that 
$d(1,\al(n) (x)) \to \infty$ as $n \to \infty$. This can be relaxed immediately as follows. We can assume that all the $\psi_i$'s have train-track representatives with the same underlying graph. With this assumption, it follows that if the element $x$ in Lemma~\ref{t infinity0} is represented by an immersed loop $\sigma$, then so is $\Psi (\sigma)$ for any $\Psi$ corresponding to a positive word in $Q$. This allows the conclusion of Lemma~\ref{t infinity0}, and hence that of Theorem~\ref{thm-endo} to go through.
\end{remark}

\subsection{Revisiting the examples of Baker and Riley}
We quickly recall Baker--Riley's small cancellation example from \cite{baker-riley}.
Suppose $F(c_1,c_2)$ and $F(d_1,d_2)$ are two free groups of rank 2 with free bases $\{c_1,c_2\}$ and $\{d_1,d_2\}$ respectively. Suppose
\begin{eqnarray*}
	C&=&c_1c_2c_1c^2_2c_1c^3_2\cdots c_1c^r_2\\
	C_i&=&c_1c^{ri+1}_2c_1c^{ri+2}_2c_1c^{ri+3}\cdots c_1c^{ri+r}_2\\
	D_j&=&d_1d^{rj+1}_2d_1d^{rj+2}_2d_1d^{rj+3}\cdots d_1d^{rj+r}_2\\
	D_{ij}&=&d_1d^{r(il+j)+1}_2d_1d^{r(il+j)+2}_2d_1d^{r(il+j)+3}_2\cdots d_1d^{r(il+j)+r}_2
\end{eqnarray*}

where $r\ge17,~i,j\in\{1,2\}$. Let
\begin{eqnarray*}
	G_{cd}&:=&\langle c_1,c_2,d_1,d_2|c_id_jc^{-1}_i=D_{ij},~1\le i,j\le 2\rangle,\\
	G_{bcd}&:=&\langle G_{cd},b|bc_ib^{-1}=C_i,~1\le i\le 2\rangle\text{ and}\\
	G&:=&\langle G_{bcd},a|aba^{-1}=bC^{-1},~ad_ja^{-1}=bD_jb^{-1},~1\le j\le 2\rangle .\,
\end{eqnarray*}

\begin{theorem}\textup{(\cite[Theorem 1]{baker-riley})}\label{baker-riley thm}
	The groups defined above have the following properties:
	\begin{enumerate}
		\item $F(d_1,d_2)$ is a subgroup of $G_{cd}$.

		\item $F(c_1,c_2)$ is also a subgroup of $G_{cd}$ and $F(c_1,c_2)\cap F(d_1,d_2)=\{1\}$.

		\item $G_{bcd}$ is an HNN extension of $G_{cd}$ with stable letter $b$ and defining monomorphism $F(c_1,c_2)\map G_{bcd}$ sending $c_i\mapsto C_i$ for $i=1,2$.

		\item $H=\langle b,d_1,d_2\rangle< G_{bcd}$ is a free subgroup of rank $3$.

		\item $G$ is an HNN extension of $G_{bcd}$ with stable letter $a$.
	\end{enumerate}
Then, $(H,G)$ is not a Cannon--Thurston pair.
\end{theorem}


We shall first show the following instead. In \cite[Remark 8]{baker-riley}, the authors assert that `more elaborate versions' of their argument for Theorem~\ref{baker-riley thm} proves Theorem~\ref{thm-corebr}.
Since this follows in a straightforward way from Theorem~\ref{thm-endo}, and since no published proof exists, we include a proof below. The pair
$(H,G_{bcd})$ is clearly contained in the pair $(H,G)$. Thus,
Theorem~\ref{thm-corebr} directly establishes the non-existence of Cannon--Thurston maps in the simpler constituents examples
building up the example in Theorem~\ref{baker-riley thm} above.

\begin{theorem}\label{thm-corebr} With notation as in Theorem~\ref{baker-riley thm}, the pair $(H,G_{bcd})$ admits only a ray Cannon--Thurston map.
\end{theorem}

\begin{proof}
Note that $G_{cd}$ is a multiple ascending HNN extension of the free group $\langle d_1,d_2\rangle\simeq \mathbb F_2$
with stable letters $c_1,c_2$. Since the presentation is $C^\prime(1/6)$,
$G_{cd}$ is hyperbolic. Thus, in the setup of Theorem~\ref{thm-endo},
\begin{enumerate}
\item $\langle d_1,d_2\rangle\simeq \mathbb F_2$ takes the place of $K$.
\item $G_{cd}$ takes the place of $H$.
\item  $\langle c_1,c_2\rangle\simeq \mathbb F_2$  takes the place of $F$.
\item Since $G_{bcd}$  is an HNN extension of $G_{cd}$ using a
positive endomorphism on the free group $\langle c_1,c_2\rangle$,  $G_{bcd}$ takes
the place of $G$.
\item $H = \langle d_1,d_2, b\rangle$ takes the place of $G_1$.
\end{enumerate}
The theorem is now an immediate consequence of Theorem~\ref{thm-endo}.
\end{proof}

We now leverage the techniques of Theorem~\ref{thm-corebr} to prove Theorem~\ref{baker-riley thm}.

\begin{proof}[Proof of Theorem~\ref{baker-riley thm}:]
To recover Theorem~\ref{baker-riley thm} we consider the following words following notation as in the theorem. Let
\begin{enumerate}
\item $u_{n,i} = b^n c_i b^{-n}$, $i=1, 2$. Note that $u_{n,i} \in F(c_1, c_2)$ for all $n$ with length in $F(c_1, c_2)$ growing exponentially in $n$,
\item $w_{n,i,j} = u_{n,i} d_j u_{n,i}^{-1}$, $i=j, 2$. Note that
$w_{n,i,j}  \in F(d_1, d_2)$ for all $n$ with length in $F(d_1, d_2)$ growing doubly exponentially in $n$,
\item Set $$u_n=u_{n,1}=b^n c_1 b^{-n},$$ and $$w_n=w_{n,1,1}=
b^n c_1 b^{-n} d_1 b^n c_1^{-1} b^{-n}$$ for convenience of exposition.
\end{enumerate}

We then note that $u_n, w_n$ are Dehn-reduced words, and hence give word geodesics $[1,u_n], [1,w_n]$ in
$\Gamma_G$. Further, as $n \to \infty$, $u_n, w_n$, regarded as elements of $G$ converge to $b^\infty \in \partial G$, since they have initial geodesic segments of the form   $[1,b^n]$. Also, the Gromov inner product
$\langle b^\infty, w_n \rangle_1 \to \infty$ as $n\to \infty$ (since $\langle b^\infty, w_n \rangle_1 \geq n- \delta$, where $\delta$ is the hyperbolicity constant of $\Gamma_{G}$). Next, note that, as an element of $H$, $w_n \to \xi \in \partial F(d_1,d_2) \subset \partial H$. In particular, $\xi \neq b^\infty$. Observe that $b^\infty \in \partial G$ is necessarily a conical limit point of $H$. We can now apply the JKLO criterion Proposition~\ref{prop-jklo} to conclude that $(H,G) $ does not admit a Cannon--Thurston map.
\end{proof}

\subsection{Questions on a connection to holomorphic dynamics}\label{sec-cxdyn} The original motivation for the study of Cannon--Thurston maps
comes from the study of Kleinian groups acting on the Riemann sphere.
It is known \cite{mahan-annals,mahan-kl} that for finitely generated Kleinian groups, Cannon--Thurston maps  exist unconditionally and that connected limit sets  are locally connected. In particular, for any simply degenerate Kleinian surface group $G < PSl(2,\C)$, there is a single domain of discontinuity $D$ biholomorphic to the unit disk, such that
\begin{enumerate}
\item every radial ray in $D$, starting at the origin $o$ say, lands on the limit set $\Lambda_G \subset 
\hat{\C}$,
\item further, the landing point on $\Lambda_G$ depends continuously on the angle that a radial ray makes with a reference ray.
\end{enumerate}

Thus, there is no analog of the phenomenon described in this paper (non-existence of Cannon--Thurston maps) in the world of Kleinian groups. On the other hand  holomorphic dynamics on $\hat{\C}$ provides a more flexible framework. We start with a toy example.

\begin{example}\label{eg-sine}
Let $\mathbb G$ denote the closure of the topologist's sine curve, i.e.\ the closure of the set $\{(x, \sin 1/x): x \in (0,1]\}$ in the plane $\C$. Let $U = \hat{\C} \setminus \mathbb G$. Let $\D \to U$ denote the Riemann uniformizing map sending $0$ to $\infty$. Then it is easy to check that
the image in $U$ of every radial ray in $\D$ lands on $\mathbb G$. However,  the landing point does not depend continuously on the radial angle at $0 \in \D$. Else, $\mathbb G$ would be locally connected by the Caratheodory extension theorem. It is not.  
\end{example}

The dynamics of transcendental functions thus provide examples illustrating the analog of the existence of only ray Cannon--Thurston maps. In \cite{devaney-krych84,devaney-tangerman86}, Devaney and his collaborators show that  for certain transcendental
holomorphic functions, every ray lands, but the Julia set is not locally connected. In particular, the parametrization in terms of angle is 
\emph{not continuous}. 
In fact, for $f(z) = \lambda e^z$, with $\lambda < 1/e$, 
\cite[p. 50]{devaney-krych84},  the Julia set  is a Cantor bouquet of curves:  each
 curve is accumulated on by other curves. Further, the  tips 
 of these curves are accessible from infinity, so that all rays land.
 
To the best of our knowledge, the following analog for polynomial dynamics is unknown.
 
 \begin{question}\label{qn-cxdyna}
 Does there exist a polynomial function $f: \hat{\C} \to \hat{\C}$ with connected Julia set $J(f)$ such that
 the following hold.
 \begin{enumerate}
 \item Every ray from infinity lands on $J(f)$, but
 \item $J(f)$ is \emph{not locally connected}; equivalently, the landing point does \emph{not} depend continuously   on the radial angle (in the sense of Example~\ref{eg-sine}).
 \end{enumerate}
 \end{question}
 
 On the other hand, a host of examples of non-locally connected Julia sets exist for polynomials $f$.  For all the known examples, there exists some ray $r$ from infinity such that the set of accumulation points of $r$ on the Julia set $J(f)$ is a non-trivial compactum, i.e.\ it contains more than one point.
 
 In the light of Theorem~\ref{thm-omni}, the analogous phenomenon is
 unknown in the world of hyperbolic groups:
 
 \begin{question}\label{qn-rct}
 Does there exist a hyperbolic subgroup $G_1$ of a hyperbolic group $G$ and a geodesic ray $[1,\xi)_{G_1} \subset \Gamma_{G_1}$ such that the collection of accumulation points of $[1,\xi)_{G_1}$ in $\pa G$ consists of more than one point?
 \end{question}
 
 In the light of Theorem~\ref{thm-omni}, Question~\ref{qn-rct} refines
 Question~\ref{qn-main}. In short, as Questions~\ref{qn-cxdyna} and ~\ref{qn-rct} illustrate, such observed phenomena in rational holomorphic dynamics and hyperbolic groups are complementary in nature. The intersection of these two worlds may be regarded as the world of Kleinian groups, where neither of these exotic phenomena occur by the main theorems of \cite{mahan-annals,mahan-kl}. In a sense, therefore, the present paper introduces  a new line in the Sullivan dictionary--this time by broadening the Kleinian side to include the study of hyperbolic groups and their subgroups.

\appendix


\section{Other Examples and Constructions}\label{sec-others}



\subsection{On the Matsuda-Oguni construction }\label{sec-mo}
The forward direction of the theorem below is due to Matsuda and Oguni \cite{mats-oguni}. 	For the proof of the reverse implication, we need to use some standard notions and results from the theory of relatively
hyperbolic groups and spaces. For details we refer the reader to \cite{farb-relhyp} or \cite[Section 2.5]{dahmani-mj-height}.

    \begin{theorem}\label{matsuda thm}
Suppose $G=H_1*_KH_2$ is a free product with amalgamation such that $H_1$, $H_2$ are hyperbolic, and $K$ is infinite, almost malnormal and quasiconvex in $H_1$. (Hence $K$ is hyperbolic.) Then
    $(H_1,G)$ is a Cannon--Thurston pair if and only if $(K,H_2)$ is a Cannon--Thurston pair.
    \end{theorem}
    
    \begin{proof}
    \noindent First of all, we note that $G$ is hyperbolic by Lemma \ref{mat-oguni elaborate}.
    
    We first give a proof of the forward direction. This is due to Matsuda-Oguni \cite{mats-oguni} and we include it for completeness.	Since $K$ is quasiconvex in $H_1$ and the composition of Cannon--Thurston maps is a Cannon--Thurston map, we have that $(K,G)$ is a Cannon--Thurston pair.
	
	Now suppose that $(K,H_2)$ is not a Cannon--Thurston pair. Let $\{k_n\},\{k'_n\}\sse K$ be such that $\lim^K_{n\map\infty}k_n=\lim^K_{n\map\infty}k'_n\in\pa K$ but $\lim^{H_2}_{n\map\infty}k_n\ne\lim^{H_2}_{n\map\infty}k'_n$ in $\pa H_2$. By Lemma \ref{mat-oguni elaborate}, $H_2$ is quasiconvex in $G$, so we have  $\lim^G_{n\map\infty}k_n\ne\lim^G_{n\map\infty}k'_n$ in $\pa G$.
	On the other hand, $(K,G)$ is a Cannon--Thurston pair, so $\lim^G_{n\map\infty}k_n=\lim^G_{n\map\infty}k'_n\in\pa G$. This is a contradiction.

	Now we shall prove the reverse direction. Suppose that $(K,H_2)$ is a Cannon--Thurston pair. To show that $(H_1,G)$ is a Cannon--Thurston pair we use
Lemma~\ref{mitra's criterion}. For this we need to compare geodesics in $G$ and in $H_1$ joining 
	pairs of points of $H_1$.
	One way to do this is to look at the tree of spaces structure $\pi:X\map T$
	of $G$ given by the amalgamated free product decomposition of $G$ as $H_1 *_K H_2$ (see Subsection \ref{trees of spaces}). Let $Y$ be a vertex space corresponding to $H_1$. Fix $y_0\in Y$.
	We shall show that given $D\geq 0$ there exists $D'\geq 0$ such that for any $y,y'\in Y$, and
	\begin{enumerate}
	\item  geodesic $\al$ in $Y$ joining  $y,y'$,
	\item geodesic $\bt$ in $X$ joining  $y,y'$,
	\end{enumerate}
	 $d_X(y_0,\bt)\le D$ implies $d_Y(y_0,\al)\le D'$.

	From this tree of spaces, we deduce the following.
	\begin{enumerate}
		\item As $H_2$ is quasiconvex in $G$, by coning off various cosets of $H_2$ in $G$
		we get a hyperbolic graph say $\widehat{G}$. In this the diameter of the various cosets of
		$K$ is $2$. Thus if we cone off the cosets of $K$ too, then we get a quasi-isometric graph that we continue to denote by
		${\widehat{G}}$ (abusing notation slightly).
		\item Let us denote the graph obtained from $H_1$ by coning off the cosets of $K$ by $\widehat{H}_1$.
		Then $\widehat{H}_1$ is qi embedded in ${\widehat{G}}$.
	\end{enumerate}
	
	Now, given any $x,y\in H_1$, we can find a uniform quasigeodesic without back-tracking, say $\gamma$, in $\widehat{H}_1$
	 joining $x,y$. A de-electrification (see \cite[Section 2.5.2]{dahmani-mj-height} for instance)
	 of $\gamma$ in $H_1$ gives a uniform
	quasigeodesic, say $\alpha$, in $H_1$. By Lemma \ref{mat-oguni elaborate}, $H_2$ is quasiconvex, almost malnormal in $G$, so that $G$ is strongly hyperbolic relative to $H_2$ (\cite[Theorem 7.11]{bowditch-relhyp-publish}, \cite[Proposition 2.10]{mahan-relrig}). By a similar de-electrification of $\gamma$ in $G$, we  obtain a uniform
	quasigeodesic, say $\beta$, in $G$. Suppose $p\in \beta$ and $d_G(1, p)\leq D$. If $p\in \gamma\cap H_1$
	then $d_{H_1}(1, p)\leq D'$ for some $D'$ depending on $D$ as finitely generated subgroups of a
	finitely generated group are properly embedded. Otherwise, the point $p$ is contained in a segment
	$\beta_1$ of $\beta$ with end points $x_1, y_1\in H_1$ such that $\beta_1\subset x_1H_2$.
	In this case, one compares a geodesic joining $x_1, y_1$ in $x_1 K\subset H_1$ with $\beta_1$.
	As any geodesic in $G$ joining $1$ to $p$ must pass through $x_1K$, it follows that if  $p_1\in x_1 K$ is a point on such
	a geodesic then $d_G(1,p_1)\leq D$ and $d_G(p_1,p)\leq D$. Now, we use the fact that $(K,H_2)$ is a Cannon--Thurston pair and that $H_2$ is quasiconvex in $G$.
	Suppose $\eta$ is a proper function as in Lemma~\ref{mitra's criterion} for the pair $(K,H_2)$. (Note that the function $\eta$ is independent of the choice of base point as the  action of $K$ on the vertex set of its Cayley graph is transitive.)
	This means that if $\sigma$ is a geodesic in $x_1K$ joining $x_1, y_1$ then $d_{x_1K}(p_1, \sigma)\leq \eta(D)$.
	Thus $d_{H_1}(1, \alpha)\leq d_{H_1}(1, p_1)+ d_{H_1}(p_1, \sigma)$ is uniformly bounded in terms of $D$ and the quasiconvexity constant of 
	$H_2$  in $G$. Applying Lemma~\ref{mitra's criterion} we are through.
\end{proof}

In \cite{mats-oguni}, using the construction of Baker--Riley, Matsuda and Oguni showed that 
any non-elementary hyperbolic group $G_1$ can be embedded in a hyperbolic group
$G$ such that $(G_1,G)$ is not a Cannon--Thurston pair. We observe in this subsection that their construction goes through verbatim for many of the examples in this paper.

Let $K_0<G_0$ be hyperbolic groups such that
\begin{enumerate}
	\item $(K_0,G_0)$ is not a Cannon--Thurston pair, and
	\item  $K_0$ is a finitely generated
	free group rank $n\geq 2$.
\end{enumerate}  
Let $U$ be any non-elementary hyperbolic group. By Theorem
\ref{main malnormal thm}, there is a free subgroup $F<U$ of rank $n$ and a finite subgroup $A<U$
such that $\langle F, A\rangle\simeq F \times A$ is almost malnormal and quasiconvex in $U$. Let $W=U*_{\langle F\times A \simeq K_0\times A\rangle}G_0\times A$. Then, as a consequence of Theorem \ref{matsuda thm}, we have the following.

\begin{cor}\textup{(Matsuda-Oguni)}\label{mat-oguni exp}
	Given $U$, let $W$ be constructed as above. Then
	$W$ is a hyperbolic group and $(U,W)$ is not a Cannon--Thurston pair.
\end{cor}

For the pair $K_0< G_0$, we can choose any of the following:
\begin{enumerate}
	\item Based on Theorem~\ref{main application}: In Theorem~\ref{main application}, let $K$ be free of rank $(n-1)$, so that the group $G_1$ occurring there
	is given as $G_1 = K * \langle t \rangle$, and is therefore a free group of rank $n$. Set $K_0=G_1$. Set $G_0$ to be equal to the group $G$ occurring in Theorem~\ref{main application} when $K$ is free of rank $(n-1)$.
	\item Based on Theorem~\ref{main thm for comm}: In Theorem ~\ref{main thm for comm} let the commensurated subgroup $K$ of $H$ be free of rank $(n-1)$. The rest of the construction is as in the previous case.
	\item Based on Theorem~\ref{thm-endo}: Start with $K$ free as in Theorem~\ref{thm-endo}. Assume that $K$ is free of rank $(n-1)$.
	The rest of the construction is as in the previous two cases.
\end{enumerate} 


\subsection{Pairs from highly distorted examples}\label{distorted exp}
$\newline$
\noindent{\em Background on Distortion.}
Let $H<G$ be finitely generated groups with finite generating sets $S$ and $T$ respectively.
We say that $H$ is {\em undistorted} in $G$ if $H$ is quasi-isometrically embedded in $G$. Else it is said to be a distorted
subgroup of $G$ \cite{gromov-ai}. Distortion is measured by the {\em distortion function} $Dist^G_H:\N\map\N$
defined as follows.
$$\text{Dist}^G_H(n):=\text{max}\{d_H(1,g)|g\in H \text{ with } d_G(1,g)\le n\}.$$
The function $\text{Dist}^G_H(n)$ a priori depends on the choice of the finite generating sets. However, upon change of the generating sets
$S$ and $T$ the distortion functions obtained are equivalent in the following sense:

{\em Two functions $f,g:\N\map\N$ are said to be equivalent, and we write $f\simeq g$,
	if there exists $C\in\N$ such that
	$f(n)\le Cg(Cn+C)+Cn+C$ and $g(n)\le Cf(Cn+C)+Cn+C$ for all $n\in \mathbb N$.}

We note that $H<G$ is undistorted if and only if the corresponding distortion functions are linear. Some of the most sophisticated collections of distorted hyperbolic subgroups of hyperbolic groups appear in 
\cite{brady-riley-hydra} and recent work of  Dani-Riley \cite{dani-riley-fracdis}).
The following result of theirs is philosophically related to Lemma~\ref{mat-oguni elaborate}.

\begin{theorem}\textup{(\cite[Theorem 16.2]{dani-riley-fracdis})}\label{distortion amalgam}
	Suppose $G=A*_CB$ is an amalgamation of finitely generated groups. Let $f:\N\map\N$ be a function such that $n\le f(n)$ for all $n\in\N$. Further, suppose $C$ is qi embedded in $A$ and Dist$^B_C\simeq f$. Then Dist$^G_A\simeq f$.
\end{theorem}

Theorem \ref{distortion amalgam} and Corollary \ref{mat-oguni exp} give the following immediate application.
\begin{prop}\label{given distortion}
	Suppose $F<H$ are non-elementary hyperbolic groups, and $(F,H)$ is not a Cannon--Thurston pair with $\text{Dist}^H_F\simeq f$ for some function $f:\N\map\N$. Moreover, suppose $F$ is a free group of finite rank. Then given any non-elementary hyperbolic group $G_1$ there is a hyperbolic group $G_1<G$ such that $\text{Dist}^{G}_{G_1}\simeq f$ and $(G_1,G)$ is not a Cannon--Thurston pair.
\end{prop}

In \cite[pp. 159]{mitra-trees} examples of free subgroups of hyperbolic groups were constructed exhibiting an (at least) iterated 
exponential distortion function. See also \cite{bbd-cat-1disto} where it was shown that the lower bound established in \cite[pp. 159]{mitra-trees} is, in fact, the distortion function for many examples.
We recall the construction quickly.

\begin{example}\cite[pp. 159]{mitra-trees}\label{eg-highdist}
	Let $H$ be a hyperbolic group of the form $P\rtimes F$ where $P$ and $F$ are free group of rank $\ge3$. Such  examples may be found in \cite{BFH-lam}. Note that $P$ is normal and $F$ is malnormal in $H$ (see Lemma~\ref{malnormal basic}). 
	
	Consider $n$ distinct copies of $H$, say $H_1,H_2,\cdots,H_n$. Let $F_{i,1}$ and $F_{i,2}$ denote  respectively the normal subgroup and malnormal subgroup in $H_i$. Suppose $$A_n=H_1*_{K_1}H_2*_{K_2}H_3*\cdots*_{K_{n-1}}*H_n$$ where $K_i$ is identified with $F_{i,2}$ in $H_i$ and with $F_{(i+1),1}$ in $H_{i+1}$.
	
	By repeated application of the Bestvina--Feighn combination theorem  \cite{BF}, we conclude that $A_j < A_n$ are hyperbolic for all $j\le n$. Then  $F_{1,1}$ has distortion in $A_n$ (at least) an iterated exponential of height $n$. (As remarked above, it was shown in \cite{bbd-cat-1disto} that the lower bound established in \cite[pp. 159]{mitra-trees} is, in fact, the distortion function  for many examples.)
\end{example}

Based on these examples, we shall construct below, for any $n\in\N$, examples of pairs of hyperbolic groups $F<H$ such that the pair $(F,H)$ admits only a ray Cannon--Thurston  map and the distortion function $Dist^H_{F}$ is an iterated exponential of height $(n+1)$.
These examples are only a placeholder illustrating the fact that armed with the \emph{techniques} of this paper 
(as opposed to its results), it is relatively easy to come up with interesting examples.


\begin{theorem}\label{distortion result}
Given $n\in\N$ and a non-elementary hyperbolic group $G_1$, we can embed $G_1$ in another hyperbolic group $G$ such that: 

$(1)$ the distortion function $Dist^{G}_{G_1}$ is exponential of height $(n+1)$, and 

$(2)$ $(G_1,G)$ is only a ray Cannon--Thurston pair.
\end{theorem}

\begin{proof}
We shall prove $(2)$ using Theorem \ref{cor-gpcriterion1}, while $(1)$ follows from the constructions of the groups (Examples \ref{eg-highdist}). Thus we shall construct groups satisfying conditions (1) and (2) of Theorem \ref{cor-gpcriterion1}, and show a limit condition (see Claim \ref{claim1} below) which ensures condition (3) of Theorem \ref{cor-gpcriterion1} under an appropriate HNN extension as required in Theorem \ref{cor-gpcriterion1}.

We now start with the collection of groups $A_n$	 described in  Example~\ref{eg-highdist}.
By repeated application of Lemmas \ref{mat-oguni elaborate} and \ref{composition mal}, we note that $H_j$ is quasiconvex in $A_j$ and malnormal in $A_j$. 
Again $(K_j,H_{j+1})$ is a Cannon--Thurston pair by \cite{mitra-ct}. So by Theorem \ref{matsuda thm}, $(H_1,A_2)$, $(A_2,A_3),\cdots,(A_{j-1},A_j)$ are Cannon--Thurston pairs. Hence $(H_1,A_j)$ is a Cannon--Thurston pair since a composition of Cannon--Thurston maps is a Cannon--Thurston map. Since $(F_{1,1},H_1)$ is a Cannon--Thurston pair \cite{mitra-ct}, it follows that $(F_{1,1},A_j)$ is a Cannon--Thurston pair. 

Again $H_n$ and $F_{n,2}$ are malnormal, quasiconvex in $A_n$ and $H_n$ respectively. So, $F_{n,2}$ is also malnormal and quasiconvex in $A_n$.
Finally, we need the following.

\begin{claim}\label{claim1}
	$\pa F_{n,2}\sse \pa i_{F_{1,1},A_n}(\pa F_{1,1})\sse\pa A_n$.
\end{claim} 

\begin{proof}
	
Suppose $\{x_l\}\sse F_{n,2}$ such that $\lim^{A_n}_{l\map\infty}x_l=\lim^{F_{n,2}}_{l\map\infty}x_l\in\pa F_{n,2}$. Since $\pa i_{F_{n,1},H_n}:\pa F_{n,1}\map\pa H_n$ is surjective by \cite{mitra-ct} and $F_{n,2}$ is quasiconvex in $H_n$, there exists a sequence $\{y_l\}\sse F_{n,1}=F_{(n-1),2}$ such that $\lim^{H_n}_{l\map\infty}y_l=\lim^{H_n}_{l\map\infty}x_l\in\pa H_n$. Again, since $H_n$ is quasiconvex in $A_n$, we have $\lim^{A_n}_{l\map\infty}y_l=\lim^{A_n}_{l\map\infty}x_l\in\pa A_n$.
	
	Note that $F_{(n-1),2}$ and $H_{n-1}$ are quasiconvex in $A_{n-1}$. Then by the same argument as above, we  have a sequence $\{y'_l\}\sse F_{(n-2),2}$ such that $\lim^{A_{n-1}}_{l\map\infty}y'_l=\lim^{A_{n-1}}_{l\map\infty}y_l\in\pa A_{n-1}$. Again, since $(A_{n-1},A_n)$ is a Cannon--Thurston pair,  $\lim^{A_n}_{l\map\infty}y'_l=\lim^{A_n}_{l\map\infty}y_l=\lim^{A_n}_{l\map\infty}x_l\in\pa A_n$.
	Then by repeated application of the above argument, we obtain a sequence $\{z_n\}$ in $F_{1,2}$ such that $\lim^{A_n}_{l\map\infty}z_l=\lim^{A_n}_{l\map\infty}x_l\in\pa A_n$.
	
	Observe that $F_{1,2}$ is quasiconvex in $H_1$ and the Cannon--Thurston map $\pa i_{F_{1,1},H_1}:\pa F_{1,1}\map H_1$ is surjective by \cite{mitra-ct}. After passing to a subsequence, if necessary, we have $\{z'_l\}\sse F_{1,1}$ such that $\lim^{H_1}_{l\map\infty}z'_l=\lim^{H_1}_{l\map\infty}z_l\in\pa H_1$. Again the Cannon--Thurston map $\pa i_{H_1,A_n}:\pa H_1\map\pa A_n$ ensures that $\lim^{A_n}_{l\map\infty}z'_l=\lim^{A_n}_{l\map\infty}z_l\in\pa A_n$. Therefore, we have a sequence $\{z'_l\}\sse F_{1,1}$ such that $\lim^{F_{1,1}}_{l\map\infty}z'_l\in\pa F_{1,1}$ and $\lim^{A_n}_{l\map\infty}z'_l=\lim^{A_n}_{l\map\infty}x_l\in\pa A_n$. This completes the proof of claim.
\end{proof}

Finally, let $\phi:F_{n,2}\map F_{n,2}$ be a hyperbolic automorphism,  so that the HNN extension $G=A_n*_{F_{n,2},\phi}$ is hyperbolic (using the Bestvina--Feighn combination theorem \cite{BF}). Let $G_1$ be the subgroup of $G$ generated by $F_{1,1}$ and $\{t\}$ where $t$ is the stable letter of the HNN extension. 

Let $L=F_{n,2}\rtimes_{\phi}\mathbb Z$. Then $G=A_n*_{F_{n,2}}L$. Note that $(A_n,G)$ is a Cannon--Thurston pair by Theorem \ref{matsuda thm} as $(F_{n,2},L)$ is a Cannon--Thurston pair (\cite{mitra-ct}). Thus by Claim \ref{claim1}, we have a sequence $\{p_l\}\sse F_{1,1}$ such that $\lim^G_{l\to\infty}p_l=t^{\infty}$ as $\pa F_{n,2}\to\pa L$ is surjective by \cite{mitra-ct}. This shows condition $(3)$ of Theorem \ref{cor-gpcriterion1}. Also note that $F_{1,1}\cap F_{n,2}=\{1\}$ by repeated application of \cite[Lemma 1]{miller-britt-thm}.
Conclusion $(2)$ of the theorem follows by Theorem \ref{cor-gpcriterion1}.

To conclude $(1)$, we note that $F_{n,2}$ is undistorted in $A_n$ and distorted exponentially in $L$. Hence by Theorem \ref{distortion amalgam}, $A_n$ is exponentially distorted in $G$, and hence, $F_{1,1}$ in $G$ has exponential distortion of height $(n+1)$. Since $G_1=F_{1,1}*\langle t\rangle$, it follows that $G_1$ has  exponential distortion of height $(n+1)$.
\end{proof}

\newcommand{\etalchar}[1]{$^{#1}$}

\end{document}